\documentclass{amsart}
\usepackage{mathrsfs}
\usepackage{blindtext}
\usepackage{amsmath,amsfonts,amsthm}
\usepackage{amssymb,latexsym}
\usepackage[colorlinks]{hyperref}
\usepackage{graphicx}
\usepackage[all,cmtip]{xy}
\usepackage{verbatim}
\usepackage{enumitem}

\usepackage{comment}
\usepackage{xcolor}

\usepackage{layout}
\usepackage{bbm}

\newtheorem{theorem}{Theorem}[section]
\newtheorem{lemma}[theorem]{Lemma}

\theoremstyle{definition}
\newtheorem{definition}[theorem]{Definition}
\newtheorem{example}[theorem]{Example}

\newtheorem{corollary}[theorem]{Corollary}
\newtheorem{proposition}[theorem]{Proposition}

\newtheorem{trm}{Theorem}

\newtheorem{crol}[trm]{Corollary}

\theoremstyle{remark}

\numberwithin{equation}{section}

\begin{document}

\title[Orbit Equivalence]{Orbit Equivalence of actions on Cartan pairs}


\author{Massoud Amini}
\address{Department of Mathematics,
	Faculty of Mathematical Sciences,
	Tarbiat Modares University, Tehran 14115-134, Iran}
\curraddr{}
\email{mamini@modares.ac.ir, mahdimoosazadeh7@gmail.com}
\thanks{}

\author{Mahdi Moosazadeh}
\address{}
\curraddr{}
\email{}
\thanks{}


\date{}

\dedicatory{}


\begin{abstract}
	We introduce and study the notion of continuous orbit equivalence of actions of countable discrete groups on Cartan pairs in  (twisted) groupoid context. We characterize orbit equivalence of actions in terms of the corresponding C$^*$-algebraic crossed products using Kumjian-Renault theory. We relate our notion to the classical notion of orbit equivalence of actions on topological spaces by showing that, under certain conditions, orbit equivalence of  actions on  (twisted) groupoids follows from orbit equivalence of  restricted actions on  unit spaces. We illustrate our results with concrete examples of continuous orbit equivalent actions on groupoids coming from odometer transformations. 
\end{abstract}

\maketitle

\section{Introduction}

Dynamical systems have ever been a fruitful source of examples for  operator algebras. The very early examples in type theory of von Neumann algebras,  given by Murray and von Neumann, used group actions on measure spaces. Later, W. Krieger  showed that for two ergodic non-singular systems, the associated von Neumann crossed product
factors are isomorphic iff the systems are orbit equivalent \cite{krie1, krie2}, building upon an earlier result of H. Dye  that all ergodic p.m.p actions are  orbit equivalent. Since then, orbit equivalence is extensively studied both in the measurable \cite{FMI} and topological \cite{GPS}, \cite{BT} setting. 

 Two p.m.p. actions $\Gamma_1 \curvearrowright X_1$ and $\Gamma_2 \curvearrowright X_2$ are called orbit equivalent if there exists an isomorphism $\phi:X_1 \rightarrow X_2$ of measure spaces sending orbits to orbits (a.e.). For discrete countable groups acting essentially free, this is known to be equivalent to the corresponding measure groupoids being isomorphic, and to the existence of an isomorphism between the corresponding von Neumann algebraic crossed products, preserving  $L^\infty$ masas \cite{Singer}, \cite{VII}.

The notion of orbit equivalence in topological setting is studied for $\mathbb{Z}^d$ actions on Cantor set in  \cite{GMPSII}, \cite{GPS}, and in general in \cite{li}.
The actions $\Gamma_1 \curvearrowright X_1$ and $\Gamma_2 \curvearrowright X_2$ of discrete groups  on topological spaces are continuous orbit equivalent if there exist a homeomorphism $\phi : X_1 \rightarrow X_2$ and continuous maps $a : \Gamma_1 \times X_1 \rightarrow \Gamma_2$ and $b : \Gamma_2 \times X_2 \rightarrow \Gamma_1$ with 
$
		\phi(\gamma_1 x_1) = a(\gamma_1,x_1) \phi(x_1)$ and $ \phi^{-1}(\gamma_2 x_2) = b(\gamma_2,x_2) \phi^{-1}(x_2)$, for $\gamma_i \in \Gamma_i$ and $x_i \in X_i$, $i = 1,2$.
The topological version of the above equivalence is known to hold with  $L^\infty$ masas replaced by $C_0$ masas \cite[Theorem 1.2]{li}.

It is natural to ask weather there is a notion of orbit equivalence for general $C^*$-dynamics and this is the main objective of present paper. As already seen in both measurable and continuous cases, the masas--or more precisely, the Cartan subalgebras--play a central role in orbit equivalence, and we rather define orbit equivalence of actions on Cartan pairs, that is, actions of countable discrete groups on $C^*$-algebras, preserving given Cartan subalgebras. This makes our approach inevitably related to twisted grroupoids, as any  Cartan pair comes from  a Weyl twisted groupoid \cite{renault-1}, and a Cartan invariant action on a separable $C^*$-algebra induces an action on the corresponding Weyl twisted groupoid \cite[Proposition 3.4]{BL-1}.  

In this paper,  we define continuous orbit equivalence (coe) of actions  on  groupoids (Definition \ref{def-coe-twisted}), and give characterizations of continuous orbit equivalence for topologically principally free actions (Definition \ref{def-topologically-prinipally-free}). 
The paper is organized as follows. In section \ref{2} we define the main notions of this paper and briefly review Kumjian-Renault theory. Section \ref{3} is devoted to proof of the main results. These results are illustrated by concrete examples in the last section.

The main results of this paper are as follows. All groupoids are locally compact, Hausdorff, second countable, and \'{e}tale (unless otherwise specified) and all groups are countable discrete. Throughout, by an  isomorphism of topological groupoids we mean an algebraic isomorphism which is also a homeomorphism. We also write $e_1$ and $e_2$ to denote the identity elements of groups $\Gamma_1$ and $\Gamma_2$, respectively.   

\begin{trm}\label{theorem-1}
	For topologically free actions $\Gamma_i \curvearrowright (G_i,\Sigma_i)$ on twisted groupoids, $i = 1,2$, consider the following statements:
	\begin{enumerate}
		
		\item $\Gamma_1 \curvearrowright (G_1,\Sigma_1) \sim_{\rm coe} \Gamma_2 \curvearrowright (G_2,\Sigma_2)$,
		
		\item there exists  $\mathbb{T}$-equivariant isomorphism $\psi: \Gamma_1 \ltimes \Sigma_1 \rightarrow  \Gamma_2 \ltimes \Sigma_2$ such that $\psi (\Gamma_1, \Sigma_1^{(0)}) = (\Gamma_2,\Sigma_2^{(0)})$ and $\psi(e_1,\sigma_1) = (e_2,\phi(\sigma_2))$, where $(\phi',\phi)$ is a twisted groupoid isomorphism between $(G_1,\Sigma_1)$ and $(G_2,\Sigma_2)$,
		
		\item there exists isomorphism $\theta :  C^*_r(\Gamma_1 \ltimes G_1, \Gamma_1 \ltimes \Sigma_1) \cong  C^*_r(\Gamma_2 \ltimes G_2,\Gamma_2 \ltimes \Sigma_2)$ such that, 
		\begin{itemize}
			\item $\theta(C_0(\Sigma_1^{(0)})) = C_0(\Sigma_2^{(0)}), $
			\item $\theta(\Gamma_1 \ltimes_r C_0(G_1^{(0)})) = \Gamma_2 \ltimes_r  C_0(G_2^{(0)}),$
			\item  $\theta(C^*_r(G_1,\Sigma_1)) = C^*_r(G_2,\Sigma_2)$.
		\end{itemize}
		
	\end{enumerate}
	Then $(1) \Leftrightarrow (2) \Rightarrow (3)$. These are equivalent for topologically principally free actions.
\end{trm}

For the case of trivial twist, we get the following corollary.

\begin{crol}
	For topologically free actions $\Gamma_i \curvearrowright G_i$ on groupoids, $i = 1,2$, consider  the following statements:
	\begin{enumerate}
		
		\item $\Gamma_1 \curvearrowright G_1 \sim_{\rm coe} \Gamma_2 \curvearrowright G_2$,
		
		\item there exists groupoid isomorphism $\psi: \Gamma_1 \ltimes G_1 \rightarrow \Gamma_2 \ltimes G_2$ such that $\psi (\Gamma_1,G_1^{(0)}) = (\Gamma_2,G_2^{(0)})$ and $\psi(e_1,G_1) = \psi(e_2,G_2)$,

		\item there exists isomorphism $\theta :  C^*_r(\Gamma_1 \ltimes G_1) \cong  C^*_r(\Gamma_2 \ltimes G_2)$ such that,
		\begin{itemize}
			\item $\theta(C_0(G_1^{(0)})) = C_0(G_2^{(0)}), $
			\item $\theta(\Gamma_1 \ltimes_r C_0(G_1^{(0)})) = \Gamma_2 \ltimes_r  C_0(G_2^{(0)}),$
			\item  $\theta(C^*_r(G_1)) = C^*_r(G_2)$.
		\end{itemize}
	\end{enumerate}
	Then $(1) \Leftrightarrow (2) \Rightarrow (3)$. These are equivalent for topologically principally free actions. 
\end{crol}

We define continuous orbit equivalence of actions on Cartan pairs via the corresponding Weyl twisted groupoids (Definition \ref{cartanpair}), and get the following characterization.
 
\begin{crol}\label{main-corollary}
	Given Cartan pairs $(A_1,B_1)$ and $(A_2,B_2)$ and Cartan invariant topologically free actions $\Gamma_1 \curvearrowright A_1$ and $\Gamma_2 \curvearrowright A_2$, consider the following statements:
	\begin{enumerate}
		
		\item $\Gamma_1 \curvearrowright (A_1, B_1) \sim_{\rm coe} \Gamma_2 \curvearrowright (A_2, B_2)$,
		
		\item there exists isomorphism $\theta :  \Gamma_1 \ltimes_r A_1 \cong \Gamma_2 \ltimes_r A_2$ such that:,
		\begin{itemize}
			\item $\theta(B_1) = B_2, $
			\item $\theta(\Gamma_1 \ltimes_r B_1) = \Gamma_2 \ltimes_r  B_2,$
			\item  $\theta(A_1) = A_2$.
		\end{itemize}
		
	\end{enumerate}
	Then $(1) \Rightarrow (2)$, and these are equivalent for topologically principally free actions.
\end{crol}

We find conditions for rigidity of continuous orbit equivalence of actions on twisted groupoids, that is conditions under which such equivalence follow from equivalence of the restricted actions on the corresponding unit spaces.

\begin{trm}\label{proposition-unitspace-oe}
	For Cartan invariant actions $\Gamma_1 \curvearrowright C_r^*(G_1,\Sigma_1)$ and $\Gamma_2 \curvearrowright C_r^*(G_2,\Sigma_2)$ with isomorphism $\psi : C_r^*(G_1,\Sigma_1) \rightarrow C_r^*(G_2,\Sigma_2)$ satisfying $\psi(C_0(G_1^{(0)})) = C_0(G_2^{(0)})$,  and   induced actions $\Gamma_1 \curvearrowright G_1^{(0)} \sim_{\rm coe} \Gamma_2 \curvearrowright G_2^{(0)}$ with corresponding cocycle maps $a$ and $b$, if,
	\begin{align}\label{condition-theorem-unitspace-oe}
		a(\gamma_1,s(g_1)) = a(\gamma_1,r(g_1)), \  \ b(\gamma_2,s(g_2)) = b(\gamma_2,r(g_2)),
	\end{align}
	for $\gamma_i\in \Gamma_i, g_i\in G_i,  i=1,2$, then $\Gamma_1 \curvearrowright G_1 \sim_{\rm coe} \Gamma_2 \curvearrowright G_2$.
\end{trm}

As an immediate consequence, we get the following further characterization of continuous orbit equivalence. 
\begin{crol}\label{main-corollary-unit-space}
	In the situation of last theorem when the twists are trivial, for continuous orbit equivalent topologically free actions $\Gamma_i \curvearrowright G_i^{(0)}$, $i = 1,2$ with cocycle maps $a$ and $b$, consider the following statements:
	\begin{enumerate}
		\item $	a(\gamma_1,s(g_1)) = a(\gamma_1,r(g_1))$ and $ b(\gamma_2,s(g_2)) = b(\gamma_2,r(g_2))$, for  $\gamma_i \in \Gamma_i$ and $g_i \in G_i$; $i=1,2$,
		\item $\Gamma_1 \curvearrowright G_1 \sim_{\rm coe} \Gamma_2 \curvearrowright G_2$,
		\item there exists isomorphism $\theta : \Gamma_1 \ltimes_r C^*_r(G_1) \rightarrow \Gamma_2 \ltimes_r C^*_r(G_2)$ such that,
		\begin{itemize}
			\item $\theta(C_0(G_1^{(0)})) = C_0(G_2^{(0)}), $
			\item $\theta(\Gamma_1 \ltimes_r C_0(G_1^{(0)})) = \Gamma_2 \ltimes_r  C_0(G_2^{(0)}),$
			\item  $\theta(C^*_r(G_1)) = C^*_r(G_2)$.
		\end{itemize}
	\end{enumerate}
	Then $(1)\Leftrightarrow(2)\Rightarrow (3)$. These are equivalent for topologically principally free actions. 
\end{crol}

\section{Continuous Orbit Equivalence}\label{2}
	An action $\Gamma \curvearrowright X$ on a topological space is called topologically free if for each $x \in X$ and non identity element $\gamma \in \Gamma$, $\{x \in X | \gamma x \neq x \}$ is dense in $X$.
		For a topological groupoid $G$, Aut$(G)$ is the group of all groupoid isomorphisms $\phi: G\rightarrow G$. By an action of a group $\Gamma$ on a groupoid $G$ we mean a group homomorphism $: \Gamma \rightarrow$ Aut$(G)$. We say that $\Gamma \curvearrowright G$ is (topologically) free if the induced action $\Gamma \curvearrowright G^{(0)}$ is (topologically) free.
	The semidirect product $\Gamma \ltimes G$ is the groupoid $\Gamma \times G$ with the following operations 
		$(\gamma,g)(\gamma',g') := (\gamma \gamma', (\gamma'^{-1} g)g'), \  (\gamma, g)^{-1} = (\gamma^{-1},\gamma g^{-1})$,
	for $\gamma,\gamma' \in \Gamma$ and $g,g'\in G$. 
	
	\begin{definition}\label{groupoid oe}
		Two actions $\Gamma_1 \curvearrowright G_1$ and $\Gamma_2 \curvearrowright G_2$ on topological (not necessarily \'{e}tale) groupoids  are said to be continuous $r$-orbit equivalent if there exist a groupoid isomorphism $\phi : G_1 \rightarrow G_2$, and continuous maps $a : \Gamma_1 \times G_1^{(0)} \rightarrow \Gamma_2$ and $b : \Gamma_2 \times G_2^{(0)} \rightarrow \Gamma_1$, with
$			\phi(\gamma_1 g_1) = a(\gamma_1,r(g_1)) \phi(g_1)$ and $\phi^{-1}(\gamma_2 g_2) = b(\gamma_2,r(g_2)) \phi^{-1}(g_2)$,
for $\gamma_i\in \Gamma_i, g_i\in G_i, i=1,2$. The continuous  $s$-orbit equivalence is defined similarly. These actions are called continuous orbit equivalent, writing $\Gamma_1 \curvearrowright G_1\sim_{\rm coe}\Gamma_2 \curvearrowright G_2$, if they are both continuous $r$- and $s$-orbit equivalent with the same cocycle maps $a$ and $b$. If moreover, the cocycle maps are  independent of their second variable, giving isomorphism of the underlying groups, the actions are called conjugate, and write $\Gamma_1 \curvearrowright G_1\sim_{\rm con}\Gamma_2 \curvearrowright G_2$.
\end{definition}

	A topological space $X$ could be regarded as a  (co-trivial) groupoid with $G^{(2)}=$ Diag$(X)=:\Delta_X$ and trivial inverse and multiplication, and  in this case, the above notion of orbit equivalence is the same as the classical one. 

Note that though we do not assume a priory algebraic conditions on maps $a$ and $b$, they are indeed not only cocycle maps in the usual sense of classical dynamics (compare with \cite[Lemma 2.8]{li}), but also in the sense of actions on groupoids. More precisely, we have the following result.

	\begin{lemma}\label{topologcally free lemma}
		Let $\Gamma_1 \curvearrowright G_1$ and $\Gamma_2 \curvearrowright G_2$ be two continuous $r$-orbit equivalent topologically free actions. Let the maps $\phi,a,b$ be as they defined in Definition \ref{groupoid oe}. For  $\gamma_i, \gamma_i' \in \Gamma_i$ and $g_i, g_i' \in G_i$,   $i = 1,2$, 
		\begin{enumerate}
			\item $a(\gamma_1,s(g_1)) = a(\gamma_1,r(g_1))$ and $b(\gamma_2,s(g_2)) = b(\gamma_2,r(g_2))$,
			
			\item when products are defined,
			\begin{itemize}
				\item $a(\gamma_1',r(g_1')) a(\gamma_1 ,r(g_1)) =a(\gamma_1'\gamma_1, r((\gamma_1^{-1}g_1') g_1)), $
				\item $b(\gamma_2',r(g_2')) b(\gamma_2 ,r(g_2)) =b(\gamma_2'\gamma_2, r((\gamma_2^{-1}g_2') g_2)), $
			\end{itemize}
			\item $ a(\gamma_1,r(g_1))^{-1} = a(\gamma_1^{-1} , r(\gamma_1 g_1^{-1}))$, \  $b(\gamma_2,r(g_2))^{-1} = b(\gamma_2^{-1} , r(\gamma_2 g_2^{-1})),$
			\item $b(a(\gamma_1,r(g_1)),\phi(g_1)) = \gamma_1$, \  $a(b(\gamma_2,r(g_2)),\phi(g_2)) = \gamma_2$.
		\end{enumerate}
	\end{lemma}
	\begin{proof}
	By topological freeness, there is a dense subset $D_i \subset G_i^{(0)}$ such that $\gamma  d = d$ forces $\gamma$ to be the identity of ${\Gamma_i}$, for $\gamma \in \Gamma_i$, $d \in D_i$,  $i = 1,2$. Since groupoids are  \'{e}tale, for an arbitrary open subset $U_i \subset G_i$, $r(U_i)$ is open and $r(U_i) \cap D_i \neq \emptyset$. Thus, $r^{-1}(D_i)$ is dense  in $G_i$,  $i=1,2$.
		For  $\gamma_i, \gamma_i' \in \Gamma_i$ and $g_i, g_i' \in G_i$,   $i = 1,2$, observe that $\phi(\gamma_1 g_1)^{-1} = (a(\gamma_1,r(g_1))\phi(g_1))^{-1}$ is equivalent to $a(\gamma_1,r(g_1^{-1}))\phi(g_1^{-1}) = a(\gamma_1,r(g_1))\phi(g_1^{-1})$,  and so by topological freeness and continuity of $a$, we get (1). By a similar argument, we also have $
		b(\gamma_2,s(g_2)) = b(\gamma_2,r(g_2))$. For (2), observe that,
		\begin{align*}
			\phi((\gamma_1'\gamma_1) ((\gamma_1^{-1}g_1') g_1)) &= \phi(\gamma_1' (g_1' (\gamma_1 g_1))) =  a(\gamma_1', r(g_1' (\gamma_1 g_1))) \phi(\gamma_1 ((\gamma_1^{-1}g_1')  g_1)) \\&= a(\gamma_1', r(g_1' (\gamma_1 g_1))) a(\gamma_1, r((\gamma_1^{-1} g_1') g_1)) \phi((\gamma_1^{-1}g_1')  g_1),
		\end{align*}
and $			\phi((\gamma_1'\gamma_1) ((\gamma_1^{-1}g_1') g_1)) = a(\gamma_1'\gamma_1, r((\gamma_1^{-1}g_1') g_1)) \phi((\gamma_1^{-1}g_1') g_1)$, and again by topological freeness, 
		\begin{align*}
			a(\gamma_1'\gamma_1, r((\gamma_1^{-1}g_1') g_1)) &= a(\gamma_1', r(g_1' (\gamma_1 g_1))) a(\gamma_1, r((\gamma_1^{-1} g_1') g_1)) \\&{=} a(\gamma_1', r(g_1')) a(\gamma_1, s((\gamma_1^{-1} g_1') g_1))= a(\gamma_1', r(g_1')) a(\gamma_1, s(g_1)) \\&{=} a(\gamma_1', r(g_1')) a(\gamma_1, r( g_1)),
		\end{align*}
		and the same for $b$. Now (3) follows  from (2) and (4) is proved similarly.  
	\end{proof}

\noindent Recall that the identity elements of $\Gamma_1$ and $\Gamma_2$ are denoted by $e_1$ and $e_2$.

	\begin{theorem}\label{theorem-orbit-equivalence-groupoid-products}
		Let $\Gamma_1 \curvearrowright G_1$, $\Gamma_2 \curvearrowright G_2$ be two topologically free actions.  The fallowing are equivalent:
		\begin{enumerate}
			\item $\Gamma_1 \curvearrowright G_1 \sim_{\rm coe} \Gamma_2 \curvearrowright G_2,$
			\item there exists groupoid isomorphism $\psi: \Gamma_1 \ltimes G_1 \rightarrow \Gamma_2 \ltimes G_2$ such that $\psi (\Gamma_1,G_1^{(0)}) = (\Gamma_2,G_2^{(0)})$ and $\psi(e_1,G_1) = (e_2,G_2).$
		\end{enumerate}
	\end{theorem}
	\begin{proof}
		$(1) \Rightarrow (2)$: Let $\phi : G_1 \rightarrow G_2$ be the isomorphism given by orbit equivalence. We claim that the maps $\psi: \Gamma_1 \ltimes G_1 \rightarrow \Gamma_2 \ltimes G_2; \ (\gamma_1,g_1)\mapsto (a(\gamma_1,r(g_1)), \phi(g_1))$, and  $\chi: \Gamma_2 \ltimes G_2 \rightarrow \Gamma_1 \ltimes G_1;\ (\eta_2,h_2)\mapsto (b(\eta_2,r(h_2)),\phi^{-1}(h_2))$, are groupoid homomorphisms. 
		Let $((\gamma_1,g_1),(\gamma_2,g_2)) \in (\Gamma_1 \ltimes G_1)^{(2)}$, then $(\gamma_2^{-1}g_1,g_2) \in G_1^{(2)}$,  $r(g_2) = s(\gamma_2^{-1}g_1) = \gamma_2^{-1}s(g_1)$, and $\gamma_2 r(g_2) = s(g_1)$. For $\gamma \in \Gamma_1$,
		\[
			a(\gamma,\gamma_2 r(g_2^{-1})) = a(\gamma,\gamma_2 s(g_2)) = a(\gamma,\gamma_2 r(g_2)) = a(\gamma,s(g_1)) = a(\gamma,r(g_1)).
		\]
		Using this and Lemma \ref{topologcally free lemma}, we get,  
		\begin{align*}
			\psi(\gamma_1,g_1)\psi(\gamma_2,g_2) &=
			(a(\gamma_1,r(g_1)), \phi(g_1))(a(\gamma_2,r(g_2)) ,\phi(g_2)) \\&= (a(\gamma_1,r(g_1))a(\gamma_2,r(g_2)),(a(\gamma_2,r(g_2))^{-1} \phi(g_1))\phi(g_2)) \\&=
			(a(\gamma_1\gamma_2, r((\gamma_2^{-1}g_1)g_2)), (a(\gamma_2^{-1}, r(\gamma_2 g_2^{-1}))\phi(g_1))\phi(g_2)) \\&=
			(a(\gamma_1\gamma_2, r((\gamma_2^{-1}g_1)g_2)),(\phi(\gamma_2^{-1}g_1)) \phi(g_2)) \\&=
			\psi (\gamma_1 \gamma_2, (\gamma_2^{-1}g_1)g_2) = \psi((\gamma_1,g_1)(\gamma_2,g_2)),
		\end{align*}
		for $\gamma_1, \gamma_2\in \Gamma_1, g_1, g_2\in G_1$, thus, $\psi$ is a groupoid homomorphism. The proof for  $\chi$ is  similar. By Lemma \ref{topologcally free lemma},
		these maps are inverse of each other, and  $\psi$ satisfies the required conditions.
		
		$(2) \Rightarrow (1)$: Let $\psi : \Gamma_1 \ltimes G_1 \rightarrow \Gamma_2 \ltimes G_2$ and $\phi : G_1 \rightarrow G_2$ be  isomorphisms such that $\psi(\Gamma_1,G_1^{(0)}) = (\Gamma_2,G_2^{(0)})$ and $\psi(e_1,g_1) = (e_2,\phi(g_1))$, for $g_1\in G_1$. Let $a: \Gamma_1 \times G_1^{(0)} \rightarrow \Gamma_2$ map $(\gamma_1,r(g_1))$ to the first component of $\psi(\gamma_1,r(g_1))$ and   $b: \Gamma_2 \times G_2^{(0)} \rightarrow \Gamma_1$ map $(\gamma_2,r(g_2))$ to the first component of $\psi^{-1}(\gamma_2,r(g_2))$. By assumption, if $\psi(\gamma_1,r(g_1)) = (a(\gamma_1,r(g_1)),g_2)$, then $g_2 \in G_2^{(0)}$. Also, $
			\psi (s(\gamma_1,r(g_1))) = s(a(\gamma_1,r(g_1)),g_2)$ implies $\psi (e_1,r(g_1)) = (e_2,g_2)
$, and $g_2 = \phi(r(g_1))$. This means that, $\psi(\gamma_1,r(g_1)) = (a(\gamma_1,r(g_1)) , \phi(r(g_1))). $
		Furthermore, 
		\begin{align*}
			\psi(\gamma_1,g_1) &= \psi(\gamma_1,r(g_1)) \psi(e_1,g_1) =(a(\gamma_1,r(g_1)),\phi(r(g_1))) (e_1,\phi(g_1)) \\&= (a(\gamma_1,r(g_1)), \phi(g_1)),
		\end{align*}
and
		\begin{align*}
		(a(\gamma_1^{-1},\gamma_1 r(g_1)), \phi(\gamma_1 g_1))&=\psi(\gamma_1,g_1^{-1})^{-1} = (a(\gamma_1,r(g_1^{-1})),\phi(g_1^{-1}))^{-1} \\&= (a(\gamma_1,r(g_1^{-1}))^{-1}, a(\gamma_1,r(g_1^{-1})) \phi(g_1)).
		\end{align*}
		Thus, $ a(\gamma_1,r(g_1)) \phi(g_1) = a(\gamma_1,s(g_1)) \phi(g_1) = \phi(\gamma_1 g_1),$ for  $\gamma_1 \in \Gamma_1$ and $g_1 \in G_1$.
		Similarly, $
		b(\gamma_2,r(g_2)) \phi^{-1}(g_2) = \phi^{-1}(\gamma_2 g_2),$
		for  $\gamma_2 \in \Gamma_2$ and $g_2 \in G_2$.
		\end{proof}

	\begin{corollary}
		Let $\Gamma_1 \curvearrowright G_1 \sim_{\rm coe} \Gamma_2 \curvearrowright G_2$ be topologically free actions then $C^*_r(\Gamma_1 \ltimes G_1 ) \cong C^*_r(\Gamma_2 \ltimes G_2)$.
	\end{corollary}
	 
	Let $G$ be a groupoid and for $u,v \in G^{(0)}$, put $G_u := s^{-1}(u)$, $G^u := r^{-1}(u)$ and $G_u^v := G_u \cap G^v$. Then $G$ is called principal if for each $u \in G^{(0)}$, $G_u^u := \{u\}$. When $G$ is a topological groupoid, it is called topologically principal if the set of units $u \in G^{(0)}$ with $G_u^u = \{u\}$ is dense in $G^{(0)}$. By a bisection we mean a subgroupoid $S$ of groupoid $G$ with restrictions $r|_S$ and $s|_S$ being injective.

	Twist over groupoids are first introduced by Alex Kumjian \cite{kumjian}. Here we recall the definition of a twisted groupoid and thier reduced C$^*$-algebras.
	\begin{definition}{(\cite{sims}, Definition 11.1.1)}\label{def-twisted-groupoid}
		Let $G$ be a locally compact, Hausdorff, second countable, \'{e}tale groupoid, and $\Sigma$ be a locally compact, Hausdorff groupoid and $\mathbb{T} \times G^{(0)}$ be a trivial group bundle with fibers $\mathbb{T}$. Then $(G,\Sigma)$ is a twisted groupoid if there is a sequence 
		\[
		\mathbb{T} \times G^{(0)}  \overset{\iota}{\rightarrowtail} \Sigma \overset{\pi}{\twoheadrightarrow} G
		\]
		such that the following conditions are satisfied:
		\begin{enumerate}
			\item $\iota$ and $\pi$ are groupoid morphisms with $\iota$ one-to-one and $q$ onto,
			\item $\iota$ and $\pi$  by reduction respectively define homeomorphisms $\{1\}\times G^{(0)} \cong \Sigma^{(0)}$ and $\Sigma^{(0)} \cong G^{(0)}$,
			\item For  $\sigma \in \Sigma$ and $t \in \mathbb{T}$, $\iota(t,r(\sigma))\sigma = \sigma \iota(t,s(\sigma))$, 
			\item $\pi^{-1} (u)  = \iota (\mathbb{T} \times \{u\})$ for  $u \in G^{(0)}$.
		\end{enumerate}
		The twist is said to be locally trivial if in addition,
		\begin{enumerate}
			\setcounter{enumi}{4}
			\item For each $g \in G$ there is an open bisection $U_g$ containing   $g$ and  a continuous map $\phi_g : U_g \rightarrow \Sigma$ with $\pi \circ \phi_g = id|_{U_g}$, making $ (t,u) \mapsto \iota(t,r(u))\phi_g(u)$ a homeomorphism between $\mathbb{T} \times U$ and $\pi^{-1}(U)$.
		\end{enumerate}
	\end{definition}

	The particular case
	$
	\mathbb{T} \times G^{(0)}  \rightarrowtail \mathbb{T} \times G \twoheadrightarrow G
$, gives the trivial twist. 
		A twisted groupoid is called topologically principal if $G$ is so.
In this case, there exist an action of the unit circle $\mathbb{T}$ on $\Sigma$ such that when $\pi(\sigma_1) = \pi(\sigma_2)$, there is a unique $t \in \mathbb{T}$ with $t \sigma_1 = \sigma_2$ \cite[11.1.2 and 11.1.3]{sims}.

	Let $G$ be an \'{e}tale groupoid with Haar system of counting measures and $(G,\Sigma)$ be a twisted groupoid with corresponding maps $\iota$ and $\pi$. We have the $*$-algebra
	\[
	C_c(G,\Sigma) := \{ f \in C_c(\Sigma,\mathbb{C}) : f(t\sigma) = \bar{t}f(\sigma) ,\text{ for all } \sigma \in \Sigma, t \in \mathbb{T} \}
	\]
under the operations 
$
f*h(\sigma) = \sum_{\dot{\tau} \in G_{s(\sigma)}}  f(\sigma \tau^{-1}) h(\tau), \ \  f^*(\sigma) = \overline{f(\sigma^{-1})},$
For $u \in G^{(0)}$,  on the Hilbert space 
\[
\mathscr{H}_u := \{ \xi : \Sigma_u \rightarrow \mathbb{C} : \xi(t \sigma) = \bar{t} \xi(\sigma) , \sum_{\dot{\sigma} \in G_u} |\xi(\sigma)|^2 < \infty \}
\]
we define $*$-representation ${\rm Ind}\delta_u$ of $C_c(G,\Sigma)$ on $\mathscr{H}_u$ by
\[
{\rm Ind}\delta_u(f)\xi(\sigma) := \sum_{\tau \in \Sigma_u} f(\sigma \tau^{-1}) \xi(\tau),\ \ (\sigma\in \Sigma_u, \xi\in \mathscr{H}_u, f\in C_c(G,\Sigma)),
\]
and define the reduced twisted groupoid $C^*$-algebra $C^*_r(G,\Sigma)$ as the completion of $C_c(G,\Sigma)$ in the reduced norm $\| f \|_r := \text{sup}_{u \in G^{(0)}} \| {\rm Ind}\delta_u(f) \|$. Replacing sums with integrals with respect to a Haar system and a quasi-invariant measure on $G^{(0)}$, the same construction works when $G$ is not \'{e}tale \cite{renault-1}. 

		By an isomorphism $(\phi_0, \phi_1, \phi_2)$ between twisted groupoids $(G_1,\Sigma_1)$ and $(G_2,\Sigma_2)$ we mean a commutative diagram,
	\[
	\xymatrix{
		\mathbb{T} \times G^{(0)}_1  \ar[d]_-{\phi_0}  \ar[r]& \Sigma_1 \ar[d]_-{\phi_1} \ar[r] & G_1 \ar[d]^-{\phi_2}\\
		\mathbb{T} \times G^{(0)}_2  \ar[r] & \Sigma_2 \ar[r] & G_2
	}
	\]
	where $\phi_0$, $\phi_1$, and $\phi_2$ are isomorphisms. We identify $\mathbb{T} \times G_1^{(0)}$ with $\mathbb{T} \times G_2^{(0)}$ and  write $(\phi_2,\phi_1) : (G_1,\Sigma_1) \rightarrow (G_2,\Sigma_2)$ for this isomorphism. In this case, we simply say that $\phi_1$ is $\mathbb T$-equivariant.  
		
		By an action of a group $\Gamma$ on a twisted groupoid $(G,\Sigma)$ we mean two actions $\Gamma \curvearrowright \Sigma$ and $\Gamma \curvearrowright G$ with the following equivariance conditions,
		\begin{align*} \label{equivariant-condition}
			\iota(t ,\gamma u) = \gamma  \iota(t, u), \ \ \pi(\gamma  \sigma) = \gamma  \pi (\sigma),\ \ (\gamma \in \Gamma , t \in \mathbb{T}, \sigma \in \Sigma, u \in G^{(0)}). 
		\end{align*}

An action of $\Gamma$ on $(G,\Sigma)$ is called topologically free if the corresponding restricted action $\Gamma \curvearrowright G^{(0)}$ is topologically free in the classical sense. 

\begin{definition}\label{def-coe-twisted}
	Two actions $\Gamma_1 \curvearrowright (G_1,\Sigma_1)$ and $\Gamma_2 \curvearrowright (G_2,\Sigma_2)$ are called $r$-continuous ($s$-continuous) orbit equivalent  if there exists a $\mathbb{T}$-equivariant isomorphism $\phi : \Sigma_1 \rightarrow \Sigma_2$ such that the induced actions $\Gamma_1 \curvearrowright \Sigma_1$ and $\Gamma_2 \curvearrowright \Sigma_2$ are $r$-continuous ($s$-continuous) orbit equivalent  with respect to $\phi$. When the actions are both $r$-continuous  and $s$-continuous orbit equivalent, with the same isomorphism and cocyle maps, the actions are called continuous orbit equivalent, writing $\Gamma_1 \curvearrowright (G_1,\Sigma_1) \sim_{\rm coe} \Gamma_2 \curvearrowright (G_2,\Sigma_2)$. Finally, when $\Gamma_1 \curvearrowright \Sigma_1 \sim_{\rm con} \Gamma_2 \curvearrowright \Sigma_2$ with respect to a $\mathbb T$-equivariant isomorphism $\phi$, we say that the actions are conjugate.
\end{definition}

	\begin{proposition}\label{proposition-orbit-equivlence-twist}
		Given topologically free actions $\Gamma_i \curvearrowright (G_i,\Sigma_i)$, $i = 1,2$, such that there exists an isomorphism $(\phi',\phi) : (G_1,\Sigma_1) \rightarrow (G_2,\Sigma_2)$, consider following statements:
		\begin{enumerate}
			\item 
			
			$\Gamma_1 \curvearrowright (G_1,\Sigma_1) \sim_{\rm coe} \Gamma_2 \curvearrowright (G_2,\Sigma_2)$ with respect to $\phi$,

			\item $\Gamma_1 \curvearrowright G_1 \sim_{\rm coe} \Gamma_2 \curvearrowright G_2.$
		\end{enumerate}
		 Then $(1)$ implies $(2)$.
	\end{proposition}
	\begin{proof}
		For $i = 1,2$, let $(\iota_i,\pi_i)$ be the twist map of $(G_i,\Sigma_i)$, and consider the twist $\mathbb{T} \times G_i^{(0)} \overset{(c_{\Gamma_i},\iota_i) }{\longrightarrow} \Gamma_i \ltimes \Sigma_i \overset{(id_i,\pi_i)}{\longrightarrow} \Gamma_i \ltimes G_i$, where $c_{\Gamma_i}$ is the constant map to the identity element of $\Gamma_i$. By the definition of twisted isomorphism $(\phi',\phi)$ and equality $(\gamma_1,\sigma_1) = (\gamma_1 , r(\sigma_1)) (e_1, \sigma_1)$, we get the equivariance condition  $\psi(\gamma_1 , t\sigma_1) = t \psi(\gamma_1 , \sigma_1)$. Thus, $\psi$ induces $\psi' : \Gamma_1 \ltimes G_1 \rightarrow \Gamma_2 \ltimes G_2$ with $(\psi',\psi)$ an isomorphism of twisted groupoids. For $g_1 \in G_1$, $\sigma_1 \in \Sigma_1$ with  $\pi_1(\sigma_1) = g_1$, $\psi'(e_1,g_1) = \psi(e_1,\pi_2(\phi(\sigma_1)))$, that is,  $\psi'(e_1,G_1) \subseteq (e_2,G_2)$, and by the same argument, $\psi'(e_1,G_1) = (e_2,G_2)$. Also, for $u \in G_1^{(0)},$
		\begin{align*}
			\psi'(\gamma_1,u) & \in \psi' \circ (id_1,\pi_1)(\gamma_1, \pi_1^{-1}(u)) =  (id_2,\pi_2) \circ \psi (\gamma_1, \pi_1^{-1}(u)) \\ &\subseteq (id_2,\pi_2) (\Gamma_2 , \mathbb{T} \times \Sigma_2^{(0)}) = (\Gamma_2, \Sigma_2^{(0)}),
		\end{align*}
		thus $\psi'(\Gamma_1, \Sigma_1^{(0)}) \subset (\Gamma_2, \Sigma_2^{(0)}), $ and by the same argument, $\psi'(\Gamma_1, \Sigma_1^{(0)}) = (\Gamma_2, \Sigma_2^{(0)}) $, which completes the proof by Theorem \ref{theorem-orbit-equivalence-groupoid-products}
	\end{proof}

		In the notation  of the above proposition, $\Gamma_1 \curvearrowright \Sigma_1 \sim_{\rm coe} \Gamma_2 \curvearrowright \Sigma_2$ with respect to $\phi$ implies $\Gamma_1 \curvearrowright G_1 \sim_{\rm coe} \Gamma_2 \curvearrowright G_2$ with respect to $\phi'$.

		A sub-C$^*$-algebra $B$ of C$^*$-algebra $A$ is called a Cartan subalgebra if,
		\begin{enumerate}
			\item $B$ is a maximal abelian subalgebra of $A$,
			\item $B$ is regular, i.e., $N_B(A) := \{ n \in A: nBn^* \subseteq B \text{ and } n^*Bn \subseteq B \}$ generates $A$ as a C$^*$-algebra,
			\item there exists a faithful conditional expectation $P : A \rightarrow B$.
		\end{enumerate}
		The pair $(A,B)$ is called a Cartan pair. In this case, the existence of an approximate identity of $A$ in $B$ is  automatic  \cite[Theorem 2.6]{pitts-1}.

	We say that two Cartan pairs $(A_1, B_1), (A_2, B_2)$ are said to be isomorphic, writing $(A_1, B_1)\simeq (A_2, B_2)$, if there is a $C^*$-algebra isomorphism from $A_1$ from $A_2$ mapping $B_1$ onto $B_2$.  There is a correspondence between separable C$^*$-algebras containing a Cartan subalgebra and twisted groupoid C$^*$-algebras, given by Renault in  the following result \cite[Theorem 5.9]{renault-1}.  
	\begin{proposition}\label{Cartan-groupoid}
		For a Cartan pair $(A,B)$ with $A$ separable, there exists a twisted  Hausdorff, locally compact, second countable, topologically principal, \'{e}tale groupoid  $(G,\Sigma)$ with $(A,B) \cong (C^*_r(G,\Sigma),C_0(G^{(0)}))$.
	\end{proposition}
	
	The twisted groupoid $(G,\Sigma)$ in the above proposition is called the Weyl twisted groupoid associated to the Cartan pair $(A,B)$. Let us recall the construction of the Weyl twisted groupoid (for details, see \cite{renault-1} and \cite[Remark 2.8]{BL-1}).
		Let $X$ be the spectrum of the abelian C$^*$-algebra $B$. For each normalize $n \in N_A(B)$, by  \cite[Lemma 4.6]{renault-1}, $n^*n$ and $nn^*$ are in $B$. Put, $\text{dom}(n) := \{ x \in X ; n^*n(x) > 0 \}$ and $\text{ran}(n) := \{ x \in X ; nn^*(x) > 0 \}$. By   \cite[Proposition 4.7]{renault-1}, there is a partial homeomorphism $\alpha_n : \text{dom}(n) \rightarrow \text{ran}(n)$ such that 
			$n^* b n(x) = b(\alpha_n(x)) n^*n(x)$,  for  $x \in \text{dom}(n)$ and $b \in B$.
		Now the groupoid $G = G(B)$ is defined by,
		\[
		G := \{ (x, \alpha_n , y): n \in N_A(B), y \in \text{dom}(n) , \alpha_n(y) = x \} / \sim,
		\]
		where $(x, \alpha_n , y) \sim (x', \alpha_{m'} , y')$ iff $y=y'$ and there exist neighbourhood $U \subseteq X$ of $y$ such that $\alpha_n |_U = \alpha_{n'} |_U$. We denote the equivalence class of $(x,\alpha_b,y)$ by $[x,\alpha_n,y]$.  The unit space $G^{(0)}$ could be identified with $\{[x,\alpha_b,x] ; b \in B, x \in $ supp$(b)\}$. The twist  $\Sigma = \Sigma(B)$ is defined by,
		\[
		\Sigma := \{ (x, n , y) \in \text{ran}(n) \times N_A(B) \times \text{dom}(n): \alpha_n(y) = x \} / \approx,
		\]
		where $(x, n , y) \approx (x', n' , y')$ iff $y = y'$ and there exist $b,b' \in B$ with $b(y) , b'(y) > 0$ and $nb = n'b'$. We use the same notation for the equivalence classes. The unit space could be identified with $\{[x,b,x]: b \in B, x \in {\rm dom}(b)\}$. The map $[x,n,y] \mapsto [x,\alpha_n,y]$ yields a central extension, making $(G,\Sigma)$ a twisted groupoid with $(A,B) \cong (C^*_r(G,\Sigma),C_0(G^{(0)}))$. This extension uses the identification of $\mathbb{T} \times X$ with $\mathscr{B} := \{ [x,b,x] : b \in B, \ b(x) \neq 0 \} \subseteq \Sigma(B)$ given by,
		$	[x,b,x] \mapsto (b(x)/|b(x)| , x).
		$

	The next result ensures the uniqueness of twisted groupoid in Proposition \ref{Cartan-groupoid}.
	\begin{proposition}[Proposition 4.15 of \cite{renault-1}]\label{r-iso}
		Let $(G,\Sigma)$ be a twisted  Hausdorff, locally compact, second countable, topologically principal \'{e}tale groupoid. Let $A := C^*_r(G,\Sigma)$ and $B := C_0(G^{(0)})$. Then there is an isomorphism of twisted groupoids,
		\[
		\xymatrix{
			\mathscr{B}  \ar[d]_-{\phi_0}  \ar[r]& \Sigma(B) \ar[d]_-{\phi_1} \ar[r] & G(B) \ar[d]^-{\phi_2}\\
			\mathbb{T} \times G^{(0)}  \ar[r] & \Sigma \ar[r] & G
		}
		\]
	\end{proposition}
	
	As mentioned in \cite{renault-1}, using the canonical action $\mathbb{T} \curvearrowright \Sigma$ in the twisted groupoid $(G,\Sigma)$, one can define a complex line bundle $L := \frac{\mathbb{C} \times \Sigma}{\sim}$, where $(c,\sigma) \sim (c',\sigma')$ iff there exists $t \in \mathbb{T}$ such that $(c,\sigma) = (\bar{t}c , t\sigma)$. Let $[c,\sigma]$ be the corresponding equivalence class, and let $\omega : L \rightarrow \mathbb{C}; [c,\sigma]\mapsto c$. Then each section of $L$ after  composing with $\omega$ could be viewed as  a map $f : G \rightarrow \mathbb{C}$. For each section $f$, the open support of $f$ is  defined to be 
	${\rm supp}'(f) = \{ g \in G ; f(g) \neq 0 \}.  $
	By  \cite[Proposition 2.4.2]{renault-book}, each element of $C^*_r(G,\Sigma)$ could be viewed as a continuous section of $L$, whose open support is the image of the open support of $f \in C_c(G,\Sigma)$ by the  twist map $\pi$ (see, section 2.2 in \cite{graded-algebra}).
	With this convention, we have the  identification,
$
	C_0(G^{(0)}) = \{ f \in C^*_r(G,\Sigma) : {\rm supp}'(f) \subset G^{(0)} \},
$
	where $h \in C_0(G^{(0)})$ is identified with $f(\sigma) = \bar{t}h(x)$, if $\sigma \in \mathbb{T} \times G^{(0)}$, and $f(\sigma) = 0$, otherwise.
	This could be used  to explicitly write  the twisted groupoid isomorphism of  Proposition \ref{r-iso} (compare with  \cite[Proposition 4.7]{renault-1}): Let $n \in N_A(B)$, since $G$ is topologically principal, $S:= {\rm supp}'(n)$ is a bisection of $G$ \cite[Proposition 4.7]{renault-1}. This gives  groupoid isomorphisms $\phi_1 : \Sigma(B) \rightarrow \Sigma;\ (y,n,x) \mapsto (n(Sy) / \sqrt{n^*n(y)} , Sx)$ and  $\phi_2 : G(B) \rightarrow G; \ [y,\alpha_n,x] \mapsto Sx$. 
	
	\begin{proposition}[Proposition 3.4 of \cite{BL-1}] \label{proposition-cartan-crossed-product}
		Let $A$ be a separable C$^*$-algebra admitting a Cartan subalgebra $B \subseteq A$. For a Cartan invariant action of a countable group $\Gamma$ on $(A, B)$,  then there is an action of $\Gamma$ on $(G(B),\Sigma(B))$ such that, 
		\[
		\Gamma \ltimes_r A \cong \Gamma \ltimes_r C^*_r(G(B),\Sigma(B)) \cong C^*_r(\Gamma \ltimes G(B),\Gamma \ltimes \Sigma(B)). 
		\]
	\end{proposition}

		Since $\alpha : \Gamma \curvearrowright A$ is Cartan invariant, there is a restricted action of $\Gamma$ on $B \cong C_0(G^{(0)})$, or equivalently, on $G^{(0)}$. This guarantees that $N_B(A)$ is invariant under the action. A simple calculation shows that
		$
			\alpha_{\gamma n}(\gamma y) = \gamma \alpha_n(y).
		$
		We have the actions, $\Gamma \curvearrowright G$ and $\Gamma \curvearrowright \Sigma$ given by,
		\begin{align}\label{group action on groupoid}
			\gamma [x,\alpha_n,y] = [\gamma x , \alpha_{\gamma n} , \gamma y], \ \  \gamma[x,n,y] = [\gamma x, \gamma n , \gamma y],
		\end{align}
		which yields an action of $\Gamma$ on the twisted groupoid $(G(B),\Sigma(B))$.

		Let $(A_1,B_1)$ and $(A_2,B_2)$ be Cartan pairs, and $\psi : A_1 \rightarrow A_2$ be a Cartan invariant isomorphism. By Proposition \ref{Cartan-groupoid}, $(A_i,B_i) \cong (C^*_r(G_i,\Sigma_i),C_0(G^{(0)}))$, for $i = 1,2$. The Cartan invariance gives an isomorphism $G_1^{(0)} \cong G_2^{(0)}$, again denoted by $\psi$. For $n \in N_{A_1}(B_1)$, $x \in G_1^{(0)}$, and  $b' \in B_2$, say, $b' = \psi(b)$, for some $b \in B_1$, we have,
		\begin{align*}
			\psi(n^*) \psi(b) \psi(n)(\psi(x)) &= \psi(b)(\alpha_{\psi(n)} (\psi(x))) \psi(n^*n)(\psi(x)), \\  n^*bn(x) &= b(\psi^{-1}(\alpha_{\psi(n)} (\psi(x)))) n^*n(x), \\ b(\alpha_n(x)) n^*n(x) &= b(\psi^{-1}(\alpha_{\psi(n)} (\psi(x)))) n^*n(x). 
		\end{align*}
		Since $x \in$ dom$(n)$, $n^*n(x) > 0$, and the above equalities  hold for all $b \in B_1$, we get, $\alpha_{\psi(n)} (\psi(x)) = \psi(\alpha_n(x))$.
		Thus, there exists a twisted groupoid isomorphism
		$(\theta,\theta') : (G_1,\Sigma_1) \rightarrow (G_2,\Sigma_2)$, given by,
		\begin{equation}\label{groupoid-renault-isomorphism}
			\begin{aligned}
				&\theta([\alpha_{n_1}(x_1),\alpha_{n_1},x_1]) = [\alpha_{\psi(n_1)}(\psi(x_1)),\alpha_{\psi(n_1)},\psi(x_1)], \\ &\theta'[\alpha_{n_1}(x_1),n_1,x_1] = [\alpha_{\psi(n_1)}(\psi(x_1)),\psi(n_1),\psi(x_1)].
			\end{aligned}
		\end{equation}

\begin{definition} \label{cartanpair}
	Let $(A_1,B_1)$ and $(A_2,B_2)$ be Cartan pairs. Two Cartan invarinat actions $\Gamma_1 \curvearrowright A_1$ and $\Gamma_2 \curvearrowright A_2$ are said to be continuous orbit equivalent if the induced action on the corresponding Weyl twisted groupoids are continuous orbit equivalent. In this case, we write $\Gamma_1 \curvearrowright (A_1, B_1)\sim_{\rm coe} \Gamma_2 \curvearrowright (A_2, B_2)$, or simply $\Gamma_1 \curvearrowright A_1\sim_{\rm coe}\Gamma_2 \curvearrowright A_2$, if the Cartan subalgebras are known from the context. Two Cartan invarinat actions $\Gamma_1 \curvearrowright (A_1, B_1)$ and $\Gamma_2 \curvearrowright (A_2, B_2)$ are said to be conjugate if there is a covariant isomorphism $\phi: A_1\to A_2$ with $\phi(B_1)=B_2$. In this case, we write $\Gamma_1 \curvearrowright (A_1,B_1)\sim_{\rm con}\Gamma_2 \curvearrowright (A_2,B_2)$.
\end{definition}

We usually start with given actions on (twisted) groupoids and then lift the actions to the corresponding reduced groupoid $C^*$-algebras. If these  groupoid $C^*$-algebras happen to have Cartan subalgebras and the induced actions are Cartan invariant, then we get actions on the corresponding Weyl twisted groupoids. The next lemma follows from {\it Corollary C}.

\begin{lemma}
	Topologically principally free conjugate actions on Cartan pairs are continuous orbit equivalent.
\end{lemma}  

In practice, the original actions might be on non \'{e}tale groupoids, and so the actions could not be orbit equivalent in the level of groupoids, but they are equivalent (and indeed conjugate) in the level of $C^*$-algebras. To clarify this rather vague description of the situation, we need some definitions. 

 \begin{definition}
	Let $(G,\Sigma)$ be a (not necessarily \'{e}tale) twisted groupoid such that $C^*_r(G,\Sigma)$ has a Cartan subalgebra $C$. Let $\Gamma\curvearrowright (G,\Sigma)$ be an action such that the lifted action on $C^*_r(G,\Sigma)$ leaves $C$ invariant. Then the unique Weyl twisted (\'{e}tale) groupoid $(\tilde G,\tilde \Sigma)$ satisfying $\big(C^*_r(G,\Sigma), C\big)\simeq\big(C^*_r(\tilde G,\tilde \Sigma), C_0(\tilde{G}^{(0)})\big)$ is called an \'{e}tale realization of $(G,\Sigma)$, or more precisely, the \'{e}tale realization of $(G,\Sigma)$ with respect to $C$. If moreover, $\Gamma\curvearrowright\big(C^*_r(G,\Sigma), C\big)\sim_{\rm con}\Gamma\curvearrowright\big(C^*_r(\tilde G,\tilde \Sigma), C_0(\tilde{G}^{(0)})\big)$, then $(\tilde G,\tilde \Sigma)$ is called an equivariant \'{e}tale realization of $(G,\Sigma)$, or more precisely, the equivariant \'{e}tale realization of $(G,\Sigma)$ with respect to $C$.    
\end{definition}

The next lemma is just a restatement of transitivity of continuous orbit equivalence of actions on (\'{e}tale) twisted groupoids.

\begin{lemma}
Given Cartan invariant actions $\Gamma_i\curvearrowright (A_i, B_i)$ and $\Lambda_i\curvearrowright (D_i, C_i)$ with $\Gamma_i\curvearrowright (A_i, B_i)\sim_{\rm coe}\Lambda_i\curvearrowright (D_i, C_i)$, for $i=1,2$, we have $\Gamma_1\curvearrowright (A_1, B_1)\sim_{\rm coe}\Gamma_1\curvearrowright (A_1, B_1)$ iff $\Lambda_1\curvearrowright (D_1, C_1)\sim_{\rm coe}\Lambda_1\curvearrowright (D_1, C_1)$.  
\end{lemma}

The next corollary is an immediate consequence of the above lemma, but despite its easy proof, it is crucial in handling orbit equivalence of actions on non \'{e}tale groupoids. 

\begin{corollary}\label{key}
For $i=1,2$, let $(G_i,\Sigma_i)$ be (not necessarily \'{e}tale) twisted groupoids such that $C^*_r(G_i,\Sigma_i)$ has a Cartan subalgebra $C_i$. Let $\Gamma_i\curvearrowright (G_i,\Sigma_i)$ be an action such that the lifted action on $C^*_r(G_i,\Sigma_i)$ leaves $C_i$ invariant, and let $(G_i,\Sigma_i)$ has an equivariant \'{e}tale realization $(\tilde G_i,\tilde \Sigma_i)$ with respect to $C_i$. Then,

$(i)$ $\Gamma_1\curvearrowright (G_1,\Sigma_1)\sim_{\rm coe} \Gamma_2\curvearrowright (G_2,\Sigma_2)$ implies $\Gamma_1\curvearrowright (\tilde G_1,\tilde\Sigma_1)\sim_{\rm coe} \Gamma_2\curvearrowright (\tilde G_2,\tilde\Sigma_2)$.

$(ii)$  $\Gamma_1\curvearrowright (G_1,\Sigma_1)\sim_{\rm coe} \Gamma_2\curvearrowright (G_2,\Sigma_2)$ implies  $\Gamma_1\curvearrowright (C^*_r(G_1,\Sigma_1), C_1)\sim_{\rm coe} \Gamma_2\curvearrowright C^*_r(G_2,\Sigma_2), C_2).$  
\end{corollary}

We later use this to build concrete examples of orbit equivalent actions on Cartan pairs (Example \ref{example1}). We warn the reader that the converse implication in part $(i)$ of the above Corollary may fail in general. With the same token, the implication in part $(ii)$ may fail if the twisted groupoids $(G_i, \Sigma_i)$ happen not  to have an equivariant \'{e}tale realization.

		\section{Proof of The Main Results} \label{3}
	In this section we prove the main results of the paper. One of the main steps in proving our first main result  is to use Proposition \ref{r-iso} for $C^*_r(\Gamma \ltimes G, \Gamma \ltimes \Sigma)$. In order to do this, we need a condition to  ensure topological principality of $(\Gamma \ltimes G, \Gamma \ltimes \Sigma)$. 
	
	\begin{definition}\label{def-topologically-prinipally-free}
		An action $\Gamma \curvearrowright G$ is called principally free if
$\gamma  r(g) = s(g)$ forces $\gamma$ to be the identity of $\Gamma$ and $g$ to be a unit element in $G^{(0)}.$
		It is called topologically principally free if there is a dense subset $U \subseteq G^{(0)}$ such that the above condition holds for  each $g \in G_u$ and each $u\in U$. 
	\end{definition}
	The (topologically) principally free actions are (topologically) free.
	An action $\Gamma \curvearrowright G$ is principally free iff  $\Gamma \ltimes G$ is (topologically) principal. Note that,  this notion is defined only for actions on (topologically) principal groupoids.

	Let $(G,\Sigma)$ and $(H,E)$ be twisted groupoids. By an open embedding of $(H,E)$ in $(G,\Sigma)$, written as $(H,E) \subset^\circ (G,\Sigma)$, we mean open inclusions $G^{(0)} \subseteq H \subseteq G$ and $\Sigma^{(0)} \subset E \subset \Sigma$, such that the following diagram commutes, 
	\begin{align*}\label{diagram-open-inclusion}
		\xymatrix{
			\mathbb{T} \times G^{(0)}  \ar@{=}[d]  \ar[r]& \Sigma_1 \ar[r] & G_1 \\
			\mathbb{T} \times H^{(0)}   \ar[r] & E \ar@{^{(}->}[u] \ar[r] & H \ar@{^{(}->}[u]
		}
	\end{align*}
	
	Let $(G,\Sigma)$ be a twisted   groupoid with twist maps $\iota_0$ and $\pi_0$, and let  $\Gamma$ act on $(G,\Sigma)$. Then 
	\[
	\mathbb{T} \times G^{(0)}  \overset{\iota}{\rightarrowtail} \Gamma \ltimes \Sigma \overset{\pi}{\twoheadrightarrow} \Gamma \ltimes G
	\]
	is a twist, where $\iota(t,v):=(e, \iota_0(t,v))$, for the identity element $e$ of $\Gamma$, and $\pi(\gamma , \sigma):=(\gamma , \pi_0(\sigma))$. 
	\begin{enumerate}
		\item There exist an open embedding which embeds $(G,\Sigma)$ in $(\Gamma \ltimes G,\Gamma \ltimes \Sigma)$,
		\[
		\xymatrix{
			\mathbb{T} \times G^{(0)}  \ar@{=}[d]  \ar[r]^{\iota} & \Gamma \ltimes \Sigma \ar[r]^{\pi} & \Gamma \ltimes G \\
			\mathbb{T} \times G^{(0)}   \ar[r]_{\iota_0} & \Sigma \ar@{^{(}->}[u]_{\eta} \ar[r]_{\pi_0} & G \ar@{^{(}->}[u]_{\eta'}
		}
		\]
		where $\eta(\sigma) := (e,\sigma)$ and $\eta'(g) := (e,g)$, for $\sigma\in \Sigma, g\in G$.  By  \cite[Lemma 3.4]{BL-2}, there is a canonical embedding $C^*_r(G,\Sigma) \hookrightarrow C^*_r(\Gamma \ltimes G, \Gamma \ltimes \Sigma)$.
		
		\item There is an open embedding of $\Gamma \ltimes G^{(0)}$ with trivial twist in $(\Gamma \ltimes G,\Gamma \ltimes \Sigma)$,
		\[
		\xymatrix{
			\mathbb{T} \times G^{(0)}  \ar[r]^{\iota} & \Gamma \ltimes \Sigma \ar[r]^{\pi} & \Gamma \ltimes G \\
			\mathbb{T} \times G^{(0)}  \ar[u]_{(id_\mathbb{T} , id_G) }   \ar[r] & \mathbb{T} \times (\Gamma \ltimes G^{(0)}) \ar@{^{(}->}[u]_{\theta} \ar[r] & \Gamma \ltimes G^{(0)} \ar@{^{(}->}[u]_{(id_\mathbb{T} , id_G)}
		}
		\]
		where $\theta(t,(\gamma , v))=(\gamma , \iota_0(t,v))$. By the definition of twisted groupoid, $\pi_0^{-1}(u) = \iota_0(\mathbb{T} \times \{u\})$, for  $u \in G^{(0)}$, and $(\Gamma,\iota_0(\mathbb{T}\times G^{(0)})) = (\Gamma,\pi^{-1}(G^{(0)}))$ is open, since $\Gamma$ is discrete and $G^{(0)}$ is open. We use the notation $\Gamma \ltimes G^{(0)} \subset^\circ (G,\Sigma)$ for this open inclusion. There is a natural isomorphism between $\Gamma \ltimes_r C_0(G^{(0)})$ and $C^*_r(\Gamma \ltimes G^{(0)})$, and again  by  \cite[Lemma 3.4]{BL-2}, there is a canonical embedding $\Gamma \ltimes_r C_0(G^{(0)}) \hookrightarrow C^*_r(\Gamma \ltimes G, \Gamma \ltimes \Sigma)$.
	\end{enumerate}

		\begin{lemma}\label{lemma-subgroupoid}
		Let $(G,\Sigma)$ and $(H,E)$ be twisted groupoids with  $(G,\Sigma) \subset^\circ  (H,E)$. Then,
				\begin{enumerate}
			\item $C^*_r(H,E) \overset{\iota}{\hookrightarrow} C^*_r(G,\Sigma)$ and the restriction of $\iota$ to $C_0(H^{(0)})$ is an isomorphism onto $C_0(G^{(0)})$.
		\end{enumerate}
		If moreover the twisted groupoids are topologically principal, 
		\begin{enumerate}
			\setcounter{enumi}{1}
			\item $H(C_0(H^{(0)})) \overset{\iota_2}{\hookrightarrow} G(C_0(G^{(0)}))$ and $E(C_0(H^{(0)})) \overset{\iota_1}{\hookrightarrow} \Sigma(C_0(H^{(0)}))$, where $\iota_1[\alpha_n(x),n,x]:=[\alpha_{\iota(n)}(x),\iota(n),x]$ $\iota_2[\alpha_n(x),\alpha_n,x]:=[\alpha_{\iota(n)}(x),\alpha_{\iota(n)},x]$, 
			\item  for isomorphisms\  $(\phi_2,\phi_1) : \big(G(C_0(G^{(0)})), \Sigma(C_0(G^{(0)}))\big) \rightarrow (G,\Sigma)$ \ and $(\phi'_2, \phi'_1) : \big(H(C_0(H^{(0)})),E(C_0(H^{(0)}))\big) \rightarrow (H,E)$, defined in Proposition \ref{r-iso}, the following diagrams commute,
			\begin{align*}
				\xymatrix{
					H(C_0(H^{(0)}))  \ar[d]_-{\phi_2'}  \ar@{^{(}->}[r]^{\iota_2} & G(C_0(G^{(0)})) \ar[d]_-{\phi_2} \\
					H  \ar@{^{(}->}[r] & G
				}
				&&
				\xymatrix{
					E(C_0(H^{(0)}))  \ar[d]_-{\phi'_1}  \ar@{^{(}->}[r]^{\iota_1} & \Sigma(C_0(G^{(0)})) \ar[d]_-{\phi_1} \\
					E  \ar@{^{(}->}[r] & \Sigma
				}
			\end{align*}
		\end{enumerate}
	\end{lemma}
	\begin{proof}
		Since $(H,E) \subset^\circ (G,\Sigma)$,  the natural map $C_c(H,E) \rightarrow C_c(G,\Sigma)$ extends to an isometric homomorphism $: C^*_r(H,E) \hookrightarrow C^*_r(G,\Sigma)$, by  \cite[Lemma 3.4]{BL-2}.
		Since the  open support of an element in $C_c(H,E)$ is a subset of the open support of that  element in $C_c(G,\Sigma)$, and after identification,  $\iota|_{C_0(H^{(0)})} : C_0(H^{(0)}) \cong C_0(G^{(0)})$, which proves (1). 
		Put $A := C^*_r(G,\Sigma)$, $B := C^*_r(H,E)$, and $C:= C_0(G^{(0)}) \cong C_0(H^{(0)})$.
		For (2), let $[\alpha_n(x),\alpha_n,x] \in H(C_0(H^{(0)}))$, i.e.,  $n \in N_B(C)$, $x \in$ dom$(n)$. By  definition of  normalizers, $\iota(N_{B}(C)) = N_{A}(C) \cap \iota(B)$. On the other hand, since $\iota$ maps $C_0(H^{(0)})$ into $C_0(G^{(0)})$, dom$(\iota(n)) =$ dom$(n)$, thus $[\alpha_{\iota(n)}(x),\alpha_{\iota(n)},x] \in G(C_0(G^{(0)}))$. It is easy to  check that $\iota_2$ is a groupoid embedding, and the same for $\iota_1$.
		To prove (3), let $n \in N_B(C)$ and put $S_1 := {\rm supp}'(\iota(n)) \subseteq G$ and $S_2 = {\rm supp}'(n) \subseteq H$, then $\phi'_2([\alpha_{n}(x),\alpha_{n},x]) = S_1x$ and $\phi_2([\alpha_{\iota(n)}(x),\alpha_{\iota(n)},x]) = S_2x$. By the proof of part (1), $S_1 \subset S_2$, and   \cite[Proposition 4.7]{renault-1} implies that $S_1$ and $S_2$ are bisections with $S_1x = S_2x$,  for each $x$. Finally, 
		\begin{align*}
			\phi'_1[y,n,x] &=  (n(S_1x) / \sqrt{n^*n(x)} , S_1x)=(\iota(n)(S_2x) / \sqrt{\iota(n^*n)(x)} , S_2x)\\&=\phi_1[y,\iota(n),x],
		\end{align*} as required.
	\end{proof}
	
	\begin{lemma}\label{inv-subgroupoid}
		Let $B_i:= C^*_r(H_i,E_i) \subset C^*_r(G_i,\Sigma_i) := A_i$,  $i = 1,2$, where $(G_i,\Sigma_i)$ and $(H_i,E_i)$ are twisted topologically principal groupoids with $(H_i,E_i)\subset^\circ (G_i,\Sigma_i)$. Assume that there exists a Cartan invariant isomorphism $\psi : A_1 \rightarrow A_2$, mapping $B_1$ onto $B_2$. Then the isomorphism $(\phi',\phi) : (G_1,\Sigma_1) \rightarrow (G_2,\Sigma_2)$  induced by the Renault's canonical isomorphism (\ref{groupoid-renault-isomorphism}) maps $E_1$ and $H_1$ to $E_2$ and $H_2$, respectively, yielding a twisted groupoid isomorphism $(H_1,E_1) \cong (H_2,E_2)$.
	\end{lemma}
	\begin{proof}
		Let us identify $C_i := C_0(G_i^{(0)})$ with $C_0(H_i^{(0)})$, for $i = 1,2$. Let $(\phi_i,\phi_i') : (G_i(C_i),\Sigma_i(C_i)) \rightarrow (G_i,\Sigma_i)$ be the isomorphism defined in Proposition \ref{r-iso}, and $(\omega', \omega) : (G_1(C_1),\Sigma_1(C_1)) \rightarrow (G_2(C_2),\Sigma_2(C_2))$ be as in (\ref{groupoid-renault-isomorphism}). We have commutative diagrams,
		\begin{align*}
			\xymatrix{
				\Sigma_1(C_1)  \ar[d]_-{\phi_1'}  \ar[r]^{\omega}& \Sigma_2(C_2) \ar[d]^{\phi_2} \\
				\Sigma_1 \ar[r]_{\phi} & \Sigma_2
			}
			&&
			\xymatrix{
				G_1(C_1)  \ar[d]_-{\phi'_1}  \ar[r]^{\omega'}& G_2(C_2) \ar[d]^{\phi'_2} \\
				G_1 \ar[r]_{\phi'} & G_2
			}
		\end{align*}
		By Lemma \ref{lemma-subgroupoid}, we also have the following commutative diagrams,
		\begin{align*}
			\xymatrix{
				E_i(C_i)  \ar[d]_-{\phi_i}  \ar@{^{(}->}[r]^{\iota_i} & \Sigma_i(C_i) \ar[d]^{\phi_i} \\
				E_i  \ar@{^{(}->}[r] & \Sigma_i
			}
			&&
			\xymatrix{
				H_i(C_i)  \ar[d]_-{\phi'_i}  \ar@{^{(}->}[r]^{\iota'_i} & G_i(C_i) 	\ar[d]^{\phi'_i} \\
				H_i  \ar@{^{(}->}[r] & G_i
			}
		\end{align*}
		It is enough to observe that $\omega(E_1(C_1)) = E_2(C_2)$ and $\omega'(H_1(C_1)) = H_2(C_2)$. Let $n \in N_{B_1}(C_1) = N_A(C_1) \cap B_1$ and $x \in H_1^{(0)} = G_1^{(0)}$ with $x \in$ dom$(n)$, then $[\alpha_n(x),n,x] \in E_1(C_1)$, and by the paragraph before  Definition \ref{cartanpair}, we have,
		$
		\omega([\alpha_n(x),n,x]) =[\psi(\alpha_n(x)),\psi(n),\psi(x)] = [\alpha_{\psi(n)}(\psi(x)),\psi(n),\psi(x)].
		$
		Since $\psi$ is Cartan invariant, $\psi(n) \in N_{B_2}(C_2)$ and $\psi(x) \in G_2^{(0)}$, and in particular,  $[\alpha_{\psi(n)}(\psi(x)),\psi(n),\psi(x)] \in E_2(C_2)$. Similarly, $\omega'(H_1(C_1)) = H_2(C_2)$. As the following diagram commutes, 
			\[
		\xymatrix{
			\mathbb{T} \times H^{(0)}_1 \ar[d] \ar@{=}[r] & \mathbb{T} \times G^{(0)}_1  \ar@{=}[r] \ar[d] & \mathbb{T} \times G^{(0)}_2 \ar@{=}[r] \ar[d] & \mathbb{T} \times H^{(0)}_2 \ar[d] \\
			E_1 \ar@{^{(}->}[r] \ar[d] & \Sigma_1 \ar[r]^{\phi} \ar[d] & \Sigma_2  \ar[d] & E_2 \ar@{_{(}->}[l] \ar[d]\\
			H_1 \ar@{^{(}->}[r] & G_1  \ar[r]^{\phi'} & G_2 & H_2 \ar@{_{(}->}[l]
		}
		\]
		the restrictions  $\phi|_{E_1}$ and $\phi'|_{H_1}$ induce an isomorphism:  $(H_1,E_1)\rightarrow (H_2,E_2)$.
	\end{proof}
	
		\begin{lemma}\label{lemma-crossed-product-topologically-principal}
		Let $(A,B)$ be a Cartan pair, $\Gamma \curvearrowright A$ be a Cartan invariant action, and $\Gamma \curvearrowright (G(B),\Sigma(B))$ be as  in (\ref{group action on groupoid}). Then, for $X := \hat{B}$, the following are equivalent:
		\begin{enumerate}
			\item $\Gamma \ltimes G(B)$ is topologically principal,
			\item  $\beta : \Gamma \curvearrowright G(B)$ is topologically principally free,
			\item the elements $x \in X$ with the following property form a dense subset of $X$:  $
				\alpha_n(x) = \gamma  x$ with  $n \in N_A(B)$ and $x \in {\rm dom}(n)$ holds only 
			when $\gamma$ is the identity of $\Gamma$ and $n \in B$,
			
			\item  The elements $x \in X$ with the following property form a dense subset of $X$:   $
				n^* (\gamma  b) n(x) = b(x) n^*n(x)$  with $ n \in N_A(B)$ and $ x \in {\rm dom}(n)$, for all $b \in B$,
			holds only 
			when $\gamma$ is the identity of $\Gamma$ and $n \in B$.
		\end{enumerate}
	\end{lemma}
	\begin{proof}
		$(1)$ and $(2)$ are trivially equivalent. For $g = [\alpha_n(x),\alpha_n,x] \in G(B)$, $\gamma  r(g) = s(g)$ is equivalent to $\gamma  \alpha_n(x) = x$. By the definition of $G(B)$, $g$ is in the unit space if $n \in B$. Thus $x \in U_\beta$ simply means that $\alpha_n(x) = \gamma^{-1}  x$ holds only when $\gamma$ is the identity of $\Gamma$ and $n \in B$, for  $n \in N_A(B)$ with $x \in$ dom$(n)$. Thus $(2)$ and $(3)$ are equivalent.
		To show the equivalence of $(3)$ and $(4)$, it is enough to observe that $n^* (\gamma  b) n(x) = b(x) n^*n(x)$, for  $b \in B$, if and only if  $\alpha_n(x) = \gamma^{-1}  x$, for $n \in N_A(B)$ with $x \in$ dom$(n)$. For this, let $n \in N_A(B)$, $x \in$ dom$(n)$ and $\gamma \in \Gamma$ such that $n^* (\gamma  b) n(x) = b(x) n^*n(x)$, for  $b \in B$, then,
		\[
		b(\gamma^{-1}  \alpha_n(x)) n^*n(x) = (\gamma  b)(\alpha_n(x)) n^*n(x) = n^* (\gamma  b) n(x)  = b(x) n^*n(x),  
		\]
		for  $b \in B$. Since $x \in$ dom$(n)$ and $b$ is arbitrary, $\gamma^{-1}  \alpha_n(x) = x$. Conversely, if $\gamma^{-1}  \alpha_n(x) = x$, then
		$
		b(x) n^*n(x) = b(\gamma^{-1}  \alpha_n(x))n^*n(x) = n^*(\gamma  b)n(x)
		$.
	\end{proof}

\begin{definition}
	A Cartan invariant action $\Gamma \curvearrowright A$ is called topologically principally free if the condition (4) above holds.	
\end{definition}
	
		\begin{proof}[Proof of Theorem \ref{theorem-1}]
		$(1) \Leftrightarrow (2)$: This follows directly from Definition \ref{def-coe-twisted} and Theorem \ref{theorem-orbit-equivalence-groupoid-products}.
		
		$(2) \Rightarrow (3)$: The map $\psi$ lifts to $\theta' : C_c(\Gamma_1 \ltimes G_1,\Gamma_1 \ltimes \Sigma_1) \rightarrow C_c(\Gamma_2 \ltimes G_2 , \Gamma_2 \ltimes \Sigma_2)$, defined by $\theta'(f)(\gamma_2,\sigma_2) = f(\psi^{-1}(\gamma_2,\sigma_2)) $. Using open inclusions of the  paragraph   after Definition \ref{def-topologically-prinipally-free}, one can see that $\theta'$ maps  $C_c(G_1^{(0)}, \mathbb{T} \times G_1^{(0)})$ and $C_c(G_1,\Sigma_1)$ onto $C_c(G_2^{(0)}, \mathbb{T} \times G_2^{(0)})$ and $C_c(G_2,\Sigma_2)$, and $C_c(\Gamma_1 \ltimes G_1^{(0)},\mathbb{T}\times(\Gamma_1 \ltimes G_1^{(0)}))$ onto $C_c(\Gamma_2 \ltimes G_2^{(0)},\mathbb{T}\times(\Gamma_2 \ltimes G_2^{(0)}))$, respectively. This  induces an isomorphism $$\theta : C^*_r(\Gamma_1\ltimes G_1 , \Gamma_1 \ltimes \Sigma_1) \rightarrow C^*_r(\Gamma_2 \ltimes G_2 , \Gamma_2 \ltimes \Sigma_2).$$ The embedding of trivial twist of $\Gamma_i \ltimes G_i^{(0)}$ in $(\Gamma_i\ltimes G_i, \Gamma_i \ltimes \Sigma_i)$ is an open embedding, $(G_i,\Sigma_i) \subset^{\circ} (\Gamma_i\ltimes G_i, \Gamma_i \ltimes \Sigma_i)$, and $(e_i,G_i^{(0)}) \subset^{\circ} (\Gamma_i\ltimes G_i, \Gamma_i \ltimes \Sigma_i)$, for $i = 1,2$. By the proof of \cite[Lemma 3.4]{BL-2}, given $f_i \in C_c(G_i^{(0)},\mathbb{T} \times G_i^{(0)})$, $f'_i \in C_c(\Gamma_i \ltimes G_i^{(0)}, \mathbb{T} \times \Gamma_i \ltimes G_i^{(0)})$, and $f''_i \in C_c(G_i,\Sigma_i)$, we have, $||f_i||_{C_0(G_i^{(0)})} = ||f_i||_{C^*_r(\Gamma_i\ltimes G_i, \Gamma_i \ltimes \Sigma_i)}$, $||f'_i||_{C^*_r(\Gamma_i \ltimes G_i^{(0)})} = ||f'_i||_{C^*_r(\Gamma_i\ltimes G_i, \Gamma_i \ltimes \Sigma_i)}$, and $||f''_i||_{C^*_r(G_i,\Sigma_i)} = ||f''_i||_{C^*_r(\Gamma_i\ltimes G_i, \Gamma_i \ltimes \Sigma_i)}$. Therefore, $\theta$ maps $C_0(G_1^{(0)})$, $\Gamma_1 \ltimes_r C^*_r(G_1^{(0)})=C^*_r(\Gamma_1 \ltimes G_1^{(0)})$,  and $C^*_r(G_1,\Sigma_1)$ onto $C_0(G_2^{(0)})$, $\Gamma_2 \ltimes_r C_0(G_2^{(0)})$,  and $C^*_r(G_2,\Sigma_2)$, respectively.
		
		$(3) \Rightarrow (2)$:  Put $A_i := C^*_r(\Gamma_i \ltimes G_i , \Gamma_i \ltimes \Sigma_i)$, $B_i := C^*(G_i,\Sigma_i)$, $C_i := C_0(G_i^{(0)})$, and $\Gamma_i \ltimes_r C_0(G_i^{(0)}) \cong C^*_r(\Gamma_i \ltimes G_i^{(0)})=: D_i$, for $i = 1,2$. Then, there exists an isomorphism $A_1 \cong A_2$,  mapping $B_1,C_1$, and $ D_1$ onto $B_2,C_2$, and $D_2$, respectively. The additional condition ensures that the groupoids $\Gamma_1 \ltimes G_1$ and $\Gamma_1 \ltimes G_1$ are topologically principal, and by Proposition \ref{r-iso},  there exists a twisted groupoid isomorphism $(\psi',\psi) : (\Gamma_1 \ltimes G_1, \Gamma_1 \ltimes \Sigma_1) \rightarrow (\Gamma_2 \ltimes G_2, \Gamma_2 \ltimes \Sigma_2)$. Since $(G_i,\Sigma_i) \subset^{\circ} (\Gamma_i\ltimes G_i, \Gamma_i \ltimes \Sigma_i)$,  by Lemma \ref{inv-subgroupoid}, $\psi'(e_1,G_1) = (e_2,G_2)$, and $\psi(e_1,\Sigma_1) = (e_2,\Sigma_2)$, yielding  a twisted groupoid isomorphism $(G_1,\Sigma_1) \cong (G_2,\Sigma_2)$.
		Define $\phi' := \psi'|_{(e_1,G_1)}$ and $\phi := \psi|_{(id,\Sigma_1)}$, and we have twisted groupoid ismomorphism $(\phi',\phi) : (G_1,\Sigma_1) \rightarrow (G_2,\Sigma_2)$. Finally, since $\Gamma_i \ltimes G_i^{(0)} \subset^\circ (\Gamma_i\ltimes G_i, \Gamma_i \ltimes \Sigma_i)$, $\psi$ maps $(\Gamma_1, G_1^{(0)})$ onto  $(\Gamma_2, G_2^{(0)})$.
		\end{proof}
	
		Given a Cartan pair $(A,B)=:(A, C_0(X))$, a Cartan invariant action $\Gamma\curvearrowright (A,B)$ is called topologically free if the induced action $\Gamma\curvearrowright X$ is topologically free in the classical sense. Now to see that Corollary \ref{main-corollary} follows from Proposition \ref{Cartan-groupoid} and Theorem \ref{theorem-1}, simply observe that by Proposition \ref{proposition-cartan-crossed-product}, statement (3) of Theorem \ref{theorem-1} could be rephrased as the requirement that there exists an isomorphism $\theta :  \Gamma_1 \ltimes_r C^*_r(G_1,\Sigma_1) \cong \Gamma_2 \ltimes_r C^*_r( G_2,\Sigma_2)$ with,
			\begin{itemize}
				\item $\theta(C_0(G_1^{(0)})) = C_0(G_2^{(0)}), $
				\item $\theta(\Gamma_1 \ltimes_r C_0(G_1^{(0)})) = \Gamma_2 \ltimes_r  C_0(G_2^{(0)}),$
				\item  $\theta(C^*_r(G_1,\Sigma_1)) = C^*_r(G_2,\Sigma_2).$
			\end{itemize}

	\begin{proof}[Proof of Theorem \ref{proposition-unitspace-oe}]
		Put $C_i := C_0(G_i^{(0)})$, $i = 1,2,$ and let, $$(\theta,\theta') : (G_1(C_1),\Sigma(C_1)) \rightarrow (G_2(C_2),\Sigma_2(C_2))$$ be as in (\ref{groupoid-renault-isomorphism}).
		Since the actions are Cartan invariant,  $\Gamma_i$ acts on $(G_i(C_i),\Sigma_i(C_i))$ by  (\ref{group action on groupoid}). Let $a : \Gamma_1 \times G_1^{(0)} \rightarrow \Gamma_2$ and $b : \Gamma_2 \times G_2^{(0)} \rightarrow \Gamma_1$ be continuous maps yielding  $\Gamma_1 \curvearrowright G_1^{(0)} \sim_{\rm coe} \Gamma_2 \curvearrowright G_2^{(0)}$.
		Then,
		$
			\theta(\gamma  [y , \alpha_{n} , x]) = a(\gamma , x)  \theta([y , \alpha_{n} , x]),$
		for  $\gamma \in \Gamma_1$, $\gamma' \in \Gamma_2$, $ [y , \alpha_{n} , x] \in G_1(C_1)$, and $[y' , \alpha_{n'} , x'] \in G_2(C_2)$.
		By continuity of $a$, there exists an open neighborhood $U \subset G_1^{(0)}$ of $x$ such that $a(\gamma,u) = a(\gamma,x)$, for  $u \in U$. By assumption,
		\begin{align*}\label{a-fact}
			\begin{split}
				a(\gamma,x) &= a(\gamma,u) = a(\gamma,s([\alpha_n(u) , \alpha_n , u]) ) = a(\gamma,r([\alpha_n(u) , \alpha_n , u]) ) = a(\gamma,\alpha_n(u)),
			\end{split}
		\end{align*}
		and since the actions preserve Cartan subalgebras, $\gamma_1  n_1$ and $\gamma_2  n_2$ are normalizers in $C^*_r(G_1,\Sigma_1)$ and $(G_2,\Sigma_2)$, respectively, for $\gamma_i \in \Gamma_i$ and normalizers $n_i$ in $C^*_r(G_i,\Sigma_i)$, $i=1,2$. For $x \in G_1^{(0)}$,
\[
			(a(\gamma,x) \psi(n^*n))(\psi(\gamma  x)) = \psi(n^*n)(\psi(x)) = n^*n(x).
\]
		Thus, dom$(a(\gamma,x) \psi(n)) = \psi($dom$(n))$.
		Therefore,
		\begin{align*}
			\alpha_{a(\gamma,x) \psi(n)} (\psi(\gamma u)) &= \alpha_{a(\gamma,x) \psi(n)} (a(\gamma,x) \psi(u)) = a(\gamma,x) \alpha_{\psi(n)} (\psi(u)) \\&= a(\gamma , \alpha_n(u)) \psi(\alpha_{n} (u)) = \psi(\gamma \alpha_{n} (u)) \\&=  \psi(\alpha_{\gamma n} (\gamma u)) = \alpha_{\psi(\gamma n)} (\psi (\gamma u)),
		\end{align*}
		i.e., $\alpha_{a(\gamma,x) \psi(n)} |_{\psi(\gamma  U)} = \alpha_{\psi(\gamma n)}|_{\psi(\gamma  U)}$. Finally, by the paragraph  before Definition \ref{cartanpair},
$
			\theta(\gamma  [y , \alpha_{n} , x]) = [\psi(\gamma y) , \alpha_{\psi(\gamma n)} , a(\gamma,x) \psi(x)] 
			= a(\gamma , x)  \theta([y , \alpha_{n} , x]). 
$
		Similar equality holds for $\theta^{-1}$ and $b$. 
	\end{proof}
	
\noindent Finally, by Lemma \ref{topologcally free lemma} and Theorem \ref{proposition-unitspace-oe}, we get Corollary \ref{main-corollary-unit-space}.

\section{Examples} \label{4}

In this section, we give examples of orbit equivalence in our sense which give rise to non classical orbit equivalent actions. 
Let us first see what we mean by a classical action. For a continuous  action of a discrete group $G$ on a locally compact, Hausdorff  space  $X$, we might regard $X$ as a groupoid identified with the diagonal $\Delta_X\subseteq X\times X$, then $C^*_r(\Delta_X)=C_0(X)$ with Cartan subalgebra being $C_0(X)$ itself. Conversely, a Cartan invariant  action $\Gamma\curvearrowright (A,B)$ is classical if $A=B$, in which case $X$ would be the spectrum of the commutative $C^*$-algebras $A$. In this sense, an action $\Gamma\curvearrowright (G,\Sigma)$  is non classical whenever $C^*_r(G,\Sigma)$ is non abelian.

Our strategy  to build  non classical examples is to start with a classical instance of continuous orbit equivalence with cocycle maps $a: \Gamma \times X \rightarrow \Lambda$, $b : \Lambda \times Y \rightarrow \Gamma$, and given $\gamma\in\Gamma$, choose a (the largest) open neighborhood (depending on $\gamma$) of given point $x\in X$ on which $a(\gamma, )$ is constant, and the same for $b$ and points in $Y$, and use these to build  equivalence relations on $X$ and $Y$. We then find canonical actions on these equivalence relations, which are topologically free and orbit equivalent. The main challenge is that though these equivalence relations are topologically principal (and usually amenable), they are rarely \'{e}tale. The good news is that since the actions are Cartan invariant in the level of $C^*$-algebras, by Kumjian-Renault theory, there are topologically principal, \'{e}tale (twisted) groupoids giving the same Cartan pairs. Therefore, in concrete examples the real challenge is to find these \'{e}tale realizations and find their isomorphic copies inside the groupoid $C^*$-algebras involved. We illustrate how to handle this in some concrete examples, by explicitly calculating the corresponding groupoid $C^*$-algebras and \'{e}tale realizations of the involved equivalence relations. The next lemma gives a sufficient condition for conjugacy of actions on groupoids.
\begin{lemma}\label{lemma lifted action}
	Each  continuous  action $\Gamma \curvearrowright G$ on a (not necessarily \'{e}tale) groupoid induces an action $\Gamma \curvearrowright C^*_r(G)$, which is also Cartan invariant if $G$ is topologically principle and \'{e}tale. In the latter  case, for any invariant Cartan subalgebra $B$ of $C^*_r(G)$, the induced action $\Gamma \curvearrowright G(B)$ is conjugate to $\Gamma \curvearrowright G$. 
\end{lemma}
\begin{proof}
	Let us define $\gamma f(g) := f(\gamma^{-1}  g)$, for $\gamma \in \Gamma, g\in G$, and  $f \in C_c(G)$. For $f, h 
	\in  C_c(G) $ observe that,
	\begin{align*}
	(\gamma f) * (\gamma h) (g) &= \int_G (\gamma f)(\eta) (\gamma h)(\eta^{-1}g) d\lambda^{r(g)}(\eta) \\&= \int_G f(\gamma^{-1} \eta) h((\gamma^{-1}\eta^{-1})(\gamma^{-1} g)) d\lambda^{r(g)}(\eta)\\&= \int_G f(\eta) h(\eta^{-1}(\gamma^{-1} g)) d\lambda^{r(\gamma^{-1} g)}(\eta) \\&= f*h(\gamma^{-1} g) = \gamma(f*h)(g), 
	\end{align*}
	for $g\in G$. Next let us observe that the above action is bounded (indeed an isometry) in the reduced norm, defined by
	$\|f\|_r:=\sup_{u\in G^{(0)}}\|{\rm Ind}\delta_u(f)\|$, where
	$${\rm Ind}\delta_u(f)\xi(g):=\int_G f(\eta)\xi(\eta^{-1}g)d\lambda^{r(g)}(\eta),$$
	for $g\in G_u$ and $\xi\in C_c(G_u)$. By a change of variable similar to the above calculation, 
	\begin{align*}
	{\rm Ind}\delta_u(\gamma f)\xi(g)={\rm Ind}\delta_u(f)(\gamma^{-1}\xi)(\gamma^{-1} g)=\gamma {\rm Ind}\delta_u(f)(\gamma^{-1}\xi)(g),
	\end{align*}
	thus regarding $\xi$ as an element in  $L^2(G^u, \lambda^u)$, 
	$$\|{\rm Ind}\delta_u(\gamma f)\xi\|_2=\|\gamma {\rm Ind}\delta_u(f)(\gamma^{-1}\xi)\|_2=\|{\rm Ind}\delta_u(f)(\gamma^{-1}\xi)\|_2\leq \|{\rm Ind}\delta_u(f)\|\|\xi\|_2,$$ which by a density argument gives $\|{\rm Ind}\delta_u(\gamma f)\|\leq\|{\rm Ind}\delta_u(f)\|$, leading to an equality by the same argument for $\gamma$ replaced by $\gamma^{-1}$. We may thus extend this to an action on the $C^*$-algebra $C^*_r(G)$. When the groupoid is moreover topologically principle and \'{e}tale, this latter action is also Cartan invariant, since the action $\Gamma \curvearrowright G$ keeps $G^{(0)}$ invariant and the Cartan subalgebra is nothing but $C_0(G^{(0)})$. More generally, for any Cartan subalgebra $B$ of $C^*_r(G,)$, let $G(B)$ be the corresponding Weyl groupoid, $\phi_2$ be as it defined in Proposition \ref{r-iso} then for $[x,\alpha_n,y] \in G(B)$, 
	\begin{align*}
	\phi_2(\gamma[x,\alpha_n,y]) = \phi_2[\gamma x,\alpha_{\gamma n},\gamma y] = S_{\gamma n}(\gamma x) = \gamma   \phi_2[x,\alpha_n,y].
	\end{align*}
	This simply means that the induced action on Renault-Kumjian groupoid is conjugate to the original action, as claimed.
\end{proof}

Let $\Gamma \curvearrowright X $ and $ \Lambda \curvearrowright Y$ be topologically free actions of discrete groups, with respect to maps $a: \Gamma \times X \rightarrow \Lambda$, $b : \Lambda \times Y \rightarrow \Gamma$ and homeomorphism $\phi : X \rightarrow Y$.
For $\gamma \in \Gamma$ , $\lambda \in \Lambda$, and $x \in X$ (resp.,  $y \in Y$),  define $U_x^\gamma$ (resp., $V_y^\lambda$) to be the open neighbourhood of $x$ (resp., $y$) that $a(\gamma , U_x^{\gamma}) = a(\gamma,x)$ (resp., $b(\lambda , V_y^{\lambda}) = b(\lambda,y)$). In another words, $U_x^\gamma = a^{-1}(a(\gamma,x)) \cap (\gamma , X)$ (resp., $V_y^\lambda = b^{-1}(b(\lambda,y)) \cap (\lambda , Y)$). Also, define $U_x := \bigcap_{\gamma \in \Gamma} U_x^\gamma$ and $\hat{U}_x = U_x \cap \phi^{-1}(V_{\phi(x)})$ (resp.,  $V_y := \bigcap_{\lambda \in \Lambda} V_y^\lambda$ and $\hat{V}_y = V_y \cap \phi(U_{\phi^{-1}(y)})$).

\begin{lemma}\label{example-lemma}
	In the above notation,
	
		$(i)$ $\gamma  U_x = U_{\gamma  x}$ and $\lambda  V_y = V_{\lambda  y}$,
		
		$(ii)$ $\gamma  \hat{U}_x = \hat{U}_{\gamma  x}$ and $\lambda  \hat{V}_y = \hat{V}_{\lambda  y}$,

\noindent for each  $\gamma \in \Gamma$ , $\lambda \in \Lambda$, and $x \in X, y\in Y$. 
\end{lemma}

\begin{proof}
	$(i)$ Since actions are topologically free, $a(\gamma' \gamma , x) = a(\gamma' , \gamma  x) a(\gamma , x)$ by Lemma \ref{topologcally free lemma} . Given $u \in U_x$,
	\[
	a(\gamma' , \gamma  x) = a(\gamma' \gamma , x) a(\gamma , x)^{-1} = a(\gamma' \gamma , u) a(\gamma , u)^{-1} = a(\gamma' , \gamma  u),
	\]
	for each $\gamma' \in \Gamma$. This means that, $\gamma  U_x = U_{\gamma  x}$, and  by the same argument, we have $\lambda   V_y = V_{\lambda  y}$. 
	
	$(ii)$ Assume further that  $u \in \hat{U}_x$. Let us observe that  $\gamma  u \in \phi^{-1}(V_{\phi(\gamma  x)})$. Given $\lambda \in \Lambda$,
	\[
		b(\lambda,\phi(\gamma  u)) = b(\lambda,a(\gamma,u)\phi(u)) {=}   b(\lambda,a(\gamma,u)\phi(x)) = b(\lambda,a(\gamma,x)\phi(x)) = b(\lambda,\phi(\gamma   x)),
	\]
	thus, $\phi(\gamma   u) \in V_{\phi(\gamma   x)}$, as claimed. Similarly, $\lambda  \hat{V}_y = \hat{V}_{\lambda   y}$.
\end{proof}

The next proposition is an immediate consequence of Lemma \ref{example-lemma}. 
  
\begin{proposition}\label{example-proposition}
	Let $\Gamma \curvearrowright X \sim_{\rm coe} \Lambda \curvearrowright Y$ with respect to homeomorphism $\phi : X \rightarrow Y$ and let $\hat{U}_x$ and $\hat{V}_y$ be defined as above. Then, 
	\begin{align*}
	G_X := \{ (x,u) ; x \in X , u \in \hat{U}_x \}, \ \  G_Y := \{ (y,v) ; y \in Y , v \in \hat{V}_y \}
	\end{align*}
	are groupoids with respect to operations $(l,m) (m , n) = (l,n)$ and $(l,m)^{-1} = (m,l)$. The actions $\Gamma \curvearrowright G_X$ and $\Lambda \curvearrowright G_Y$, respectively defined by $\gamma  (x,u) = (\gamma  x,\gamma   u)$ and  $\lambda  (y,v) = (\lambda  y,\lambda  v)$,  are topologically free,  and $\Gamma \curvearrowright G_X \sim_{\rm coe} \Lambda \curvearrowright G_Y$, with respect to maps $a$, $b$ and groupoid isomorphism mapping $(x,u)$ on $(\phi(x),\phi(u))$.
\end{proposition}

To illustrate the above typical construction in a concrete example, let us use the example of odometer transformations in \cite[Theorem 3.3]{li} (with the warning that there is a typo in the definition of cocycles in the proof of the above cited result). Consider supernatural number $M = \Pi_{p \in P} p^{v_p(M)}$, where $P$ is the set of all prime numbers and $\sum_p = v_p(M) = \infty$. For a natural number $m$, let $\mathbb Z_m$ be the finite cyclic group of integers mod $m$. Fix a sequence $(m_k)_k$ of natural numbers such that for all prime number $p$, $v_p(m_k) \nearrow v_p(M)$, as $k \rightarrow \infty$. Now define $\mathbb{Z}_M$ to be the inverse limit of $\mathbb{Z}_{m_k}$. Natural embedding of $\mathbb{Z}$ in $\mathbb{Z}_M$ yields an action $\mathbb{Z} \curvearrowright \mathbb{Z}_M$ which is called odometer transformation for $M$. We denote the action $\mathbb{Z} \curvearrowright \mathbb{Z}_M$ and its corresponding diagonal action $\mathbb{Z} \curvearrowright \mathbb{Z}_M\times \mathbb{Z}_M$ by $\lambda_M$ and $\lambda^d_M$, respectively.

\begin{example} [Odometer transformation]\label{example1}
	For  $l \in \mathbb{N}$ and supernatural number $L$, define $X = X_{lL} = \mathbb{Z}_{lL}$, $\tilde{X} = \tilde{X}_{lL} = l.(\mathbb{Z}_{lL})$ and $Y = \mathbb{Z}_{l} \times \tilde{X}$. Let $\alpha : \mathbb{Z} \curvearrowright X$ be the odometer transformation,    $\tilde{\alpha}: l\mathbb Z\curvearrowright \tilde X$ be the corresponding restricted action, and $\lambda_l : \mathbb{Z}_l \curvearrowright \mathbb{Z}_l$ be given by translation. Then  $\alpha \sim_{\rm coe} \lambda_l \times \tilde{\alpha}$ (c.f., the proof of  \cite[Theorem 3.3]{li}). Indeed, for homeomorphism,
		\begin{align*}
		&\phi : X = \bigsqcup_{k=0}^{l-1} k + \tilde{X} \rightarrow \mathbb{Z}_l \times \tilde{X}; \ \phi(k+x) = (k,x),\ \ (x\in \tilde X),
	\end{align*}
	with inverse  $\psi : \mathbb{Z}_l \times \tilde{X} \rightarrow X$, given by $\psi(k,x)= k+x$, for $x\in \tilde X$,  and  the cocycle maps,
	\begin{align*}
		&a : \mathbb{Z} \times X = \bigsqcup_{k=0}^{l-1} k + \mathbb{Z}_l \times \bigsqcup_{k=0}^{l-1} k + \tilde{X} \rightarrow \mathbb{Z}_l \times l\mathbb{Z}; \\
		&\big((j+h), (k+x)\big) \mapsto
		\begin{cases}
			(j,h)  & j+k < l \\
			(j,h+l) & j+k \geq l,
		\end{cases}
	\end{align*}
	and, 
	\begin{align*}
		&b : (\mathbb{Z}_l \times l\mathbb{Z}) \times (\mathbb{Z}_l \times \tilde{X}) \rightarrow  \mathbb{Z}; \ \ \big((j,h), (k,x)\big) \mapsto  
		\begin{cases}
			j+h  & j+k < l \\
			j+h-l & j+k \geq l
		\end{cases}
	\end{align*}
	and element  $t = \tilde{t}l + k_0$ in $\mathbb{Z}$,  
	\begin{align*}
		\begin{split}
			\phi(t  (k+x)) = \phi( (k_0 + \tilde{t}l)  (k+x)) &= 
			\begin{cases}
				((k+k_0) , (\tilde{t}l + x))  & k_0+k < l \\
				((k+k_0 - l) , (\tilde{t}l + l + x)) &  k_0+k \geq l
			\end{cases}
			\\&=
			\begin{cases}
				(k_0 , \tilde{t}l)  (k,x) & k_0+k < l \\
				(k_0 , \tilde{t}l + l)  (k,x) &  k_0+k \geq l
			\end{cases}
			\\&=
			a((k_0 + \tilde{t}l) , (k+x)) \phi(k+x),
		\end{split}
	\end{align*}
and,
	\begin{align*}
		\begin{split}
			\psi((k_0,\tilde{t})  (k,x)) &= 
			\begin{cases}
				(k+k_0) + (\tilde{t}l + x)  & k_0+k < l \\
				(k+k_0-l) + (\tilde{t}l + x) &  k_0+k \geq l
			\end{cases}
			\\&=
			\begin{cases}
				(k_0 + \tilde{t}l)  (k+x) & k_0+k < l \\
				(k_0 + \tilde{t}l - l)  (k+x) &  k_0+k \geq l
			\end{cases}
			\\&= b((k_0 , \tilde{t}l) , (k,x)) \phi^{-1}(k,x).
		\end{split}
	\end{align*}
	Next, for $x \in \tilde{X}$ and $k \in \{0,...,l-1\}$, by the definition of $U_{k+x}^{t}$, 
	\begin{align*}
		U_{k+x}^{k_0 + \tilde{t}l} = 
		\begin{cases}
			\{k < l - k_0\} + \tilde{X} & k_0+k < l \\
			\{k \geq l - k_0\} +\tilde{X} &  k_0+k \geq l,
		\end{cases} 
	\end{align*}
	and similarly,
	\begin{align*}
		V_{(k,x)}^{(k_0 , \tilde{t}l)} = 
		\begin{cases}
			\{k < l - k_0\} \times \tilde{X} & k_0+k < l \\
			\{k \geq l - k_0\} \times \tilde{X} &  k_0+k \geq l.
		\end{cases} 
	\end{align*}
	Thus, $U_{k+x} = \bigcap_{t \in \mathbb{Z}} U_{k+x}^{t} = k + \tilde{X}$ and $V_{(k,x)} = \bigcap_{t \in \mathbb{Z}_l \times l \mathbb{Z}} V_{(k,x)}^{t} = \{k\}\times\tilde{X}$. Similarly, $\hat{U}_{k+x} = U_{k+x}$ and $\hat{V}_{(k,x)} = V_{(k,x)}$. Consider the groupoids $G_L^l$ and $H_L^l$ given by,
	\begin{align*}
		G_L^l = \{(k+x , k + u) ; k = 0,...,l-1 , x,u \in \tilde{X}\}, \\ H_L^l = \{((k,x) , (k , u)) ; k = 0,...,l-1 , x,u \in \tilde{X}\}.
	\end{align*}
	The actions $\mathbb{Z} \curvearrowright G_L^l$ and $\mathbb{Z}_l \times l\mathbb{Z} \curvearrowright H_L^l$ are defined by,
	\begin{align*}
		&z  (k+x,k+u) = (z+k+x , z+k+u) \\
		&z  ((k,x),(k,u)) = \begin{cases}
			((k+k_0 , l\tilde{z} + x),(k+k_0 , l\tilde{z} + u)) & k_0+k < l \\
			((k+k_0-l , l\tilde{z} + x),(k+k_0-l , l\tilde{z} + u)) &  k_0+k \geq l
		\end{cases} 
	\end{align*}
	where $z = \tilde{z}l + k_0$. The above groupoids are second-countable, Hausdorff, amenable, and principal. By Proposition \ref{example-proposition}, the actions are continuous orbit equivalent and by Lemma \ref{lemma lifted action} these action could be lift to their C$^*$-algebra. In order to calculate the groupoid C$^*$-algebra, first observe $H_L^l \cong \mathbb{Z}_l \times (\tilde{X} \times \tilde{X})$, where $\mathbb{Z}_l$ is regarded as co-trivial groupoid and $(\tilde{X} \times \tilde{X})$ is the groupoid of the full equivalence relation. For any Radon measure $\mu$ of full support on $\tilde{X}$,
	$
	C^*_r(G_L^l) \cong C^*_r(H_L^l)$, both being isomorphic to $C^*_r(\mathbb{Z}_l \times (\tilde{X} \times \tilde{X})) \cong \mathbb{C}^{l} \otimes C^*_r(\tilde{X} \times \tilde{X}) \cong \bigoplus_{1}^{l} \mathscr{K}(L^2 (\tilde{X} , \mu)), $
	where the last isomorphism is well known (c.f., \cite[Example 1.54]{williams}), and the first holds by the fact that $C^*_r(G_1 \times G_2) \cong C^*_r(G_1) \otimes C^*_r(G_2)$, for locally compact, Hausdorff groupoids $G_1, G_2$ (see for instance,   \cite[Lemma 5.6]{simple-algebra}, where the same fact is stated in the amenable \'{e}tale case, but these conditions are not needed).
\end{example}

In order to identify the \'{e}tale realizations of the groupoids in the above examples via the Kumjian-Renault theory, we need the following essentially known lemma, in which we restrict ourselves to the simpler case of full equivalence relations. The first two statements are well known (see, \cite[Example 1.54]{williams}, with a warning that the cited example has some typos), but we sketch the proof, for the sake of completeness.   

\begin{lemma}\label{realization} Let $X$ be a locally compact Hausdorff space and $\mu$ be a full support Radon measure on $X$, then for each $x\in X$, the induced representation $L^x:={\rm Ind} \delta_x$ extends to an isomorphism of $C^*_r(X\times X)$ onto the algebra $\mathscr{K}(L^2(X,\mu))$ of compact operators, with Cartan subalgebra $\mathscr{D}_X$, consisting of diagonal operators $D_\lambda:={\rm diag(\lambda)}$, for $\lambda=(\lambda_k)\in c_0(\mathbb Z)$. The isomorphisms $L^x$ is independent of $x$ up to equivalence, and given $\lambda=(\lambda_k)\in c_0(\mathbb Z)$ and a countable ONB $(\xi_n)_{n\in\mathbb Z}$ of $L^2(X,\mu)$ inside $C(X)$, it maps the elements $f_\lambda\in C_0(X\times X)$ defined by,
	$$f_\lambda(x,y):=\sum_{k\in\mathbb Z} \lambda_k\xi_k(x)\bar\xi_k(y),\ \ (x,y\in X),$$  
to the diagonal operator $D_\lambda$. In particular, the $C^*$-subalgebra $\mathscr{C}_X$ generated by these continuous functions in $C^*_r(X\times X)$ is the Cartan subalgebra of $C^*_r(X\times X)$, mapped by $L^x$ isometrically onto $\mathscr{D}_X$.
\end{lemma} 
\begin{proof}
It is known that, for $x\in X$, the maps $L^x$ are all equivalent \cite[Exercise 1.4.8]{williams}, and give the reduced $C^*$-norm. Identifying $L^2(\{x\}\times X, \delta_x\times \mu)$ with $L^2(X,\mu)$, one could write $L^x$ as,
$$L^x(f)\xi(y)=\int_X f(y,z)\xi(z)d\mu(z), \ \ \big(y\in X, f\in C_c(X\times X), \xi\in L^2(X,\mu)\big),$$
which is a Hilbert-Schmidt operator with kernel $f$. In particular, $L^x$ is surjective. Since $X\times X$ is equivalent to the trivial group, $C^*_r(X\times X)$ is Morita equivalent to $\mathbb C$, by the celebrated Renault's equivalence theorem for reduced algebras  \cite[Theorem 2.71]{williams}, and so simple. This forced $L^x$ to be one-one, which is the only non trivial part of showing that it is a $C^*$-isomorphism.  

The fact that $\mathscr{D}$ is the Cartan subalgebra of $\mathscr{K}(L^2(X,\mu))$ is well known. Fix a countable ONB $(\xi_n)_{n\in\mathbb Z}$ of $L^2(X,\mu))$ inside $C(X)$, which exists by standard Fourier theory. Next, for each $n\in\mathbb Z$ and $\lambda=(\lambda_k)\in c_0(\mathbb Z)$, 
\begin{align*}
L^x(f_\lambda)\xi_n(y)&=\int_X f_\lambda(y,z)\xi(z)d\mu(z)\\
&=\sum_{k\in\mathbb Z} \lambda_k\int_X\xi_k(y)\bar\xi_k(z)\xi_n(z)d\mu(z)\\&=\sum_{k\in\mathbb Z} \lambda_k \langle\xi_n,\xi_k\rangle\xi_k(y)=\lambda_n\xi_n(y),
\end{align*} 
that is, $L^x(f_\lambda)=D_\lambda$. The last statement now follows, as $L^x$ is an isomorphism.  
\end{proof} 	
	  
\begin{lemma}\label{equivariant}
	In the notations of Example \ref{example1}, 
	
	$(i)$ there is a full support Radon measure $\mu$ on $X$ and a $\mathbb Z$-equivariant isomorphism $\theta_X: L^2(X, \mu)\to \ell^2(\mathbb Z)$, lifting to a Cartan preserving $\mathbb Z$-equivariant isomorphism $\tilde\theta_X: C^*_r(X\times X)\to C^*_r(\mathbb Z\times \mathbb Z)$.
	
	$(ii)$ there is a full support Radon measure $\nu$ on $Y$ and a $\mathbb (Z_l\times l\mathbb Z)$-equivariant isomorphism $\theta_Y: L^2(Y, \nu)\to \ell^2(\mathbb Z_l\times l\mathbb Z)$, lifting to a Cartan preserving $(Z_l\times l\mathbb Z)$-equivariant isomorphism $\tilde\theta_Y: C^*_r(Y\times Y)\to C^*_r\big((\mathbb Z_l\times l\mathbb Z)\times (\mathbb Z_l\times l\mathbb Z)\big).$   
\end{lemma}
 \begin{proof}
$(i)$ First note that $X=\varprojlim \mathbb Z_{m_k}$ is a procyclic (compact abelian) group with (discrete) Pontryagin dual $\hat X=\varinjlim \mathbb Z_{m_k}$ \cite[Lemma 2.9.3]{rz}, and the latter is an ONB for $L^2(X,\mu)$, for any normalized left Haar measure $\mu$ on $X$ \cite[Corollary 4.27]{f}.  Each element of $\hat X=\bigcup \mathbb Z_{m_k}$ is a positive integer $\ell\ ({\rm mod}\ m_k)$, which corresponds to a character $$\xi_\ell^k\big((n_j)_{j=1}^\infty\big):={\rm exp}\big(-2\pi i(\ell+n_k)/m_k\big); \ \ \big((n_j)_{j=1}^\infty\in X).$$ 
For simplicity, let us use the notation $[n]_k:= n\ ({\rm mod}\ m_k)$, for $n\in \mathbb Z$. Take a bijection between $\hat X=\{[n]_k: n\in \mathbb Z\}$ and $\mathbb Z,$ and let $n(\ell, k)$ be the image of $[\ell]_k$ under this bijection, for $\ell=0,s, m_k-1$. Let $\mathbb Z$ act on itself by translation, then for $t\in \mathbb Z$, 
$n(\ell,k)+t$ is the image of $[\ell]_k+t=[\ell+t]_k$, thus $n(\ell,k)+t=n(\ell+t,k)$.  Let $(e_n)_{n\in \mathbb Z}$ be the canonical ONB of $\ell^2(\mathbb Z)$ and $\theta_X: L^2(X, \mu)\to \ell^2(\mathbb Z)$ be the isomorphism mapping $\xi_\ell^k$ to $e_{n(\ell,  k)}$, for $k\geq 1$, $\ell=0,s, m_k-1$.  Then,
\begin{align*}
t\xi_\ell^k\big((n_j)\big)&=\xi_\ell^k\big(([n_j-t]_j)\big)\\&={\rm exp}\big(-2\pi i(\ell+[n_k-t]_k)/m_k\big)\\&={\rm exp}\big(-2\pi i(\ell+n_k-t)/m_k\big)\\&={\rm exp}\big(-2\pi i([\ell-t]_k+n_k)/m_k\big)=\xi_{[\ell-t]_k}^k\big((n_j)\big)
\end{align*} 
for $k\geq 1$, $\ell=0,s, m_k-1$, and $(n_j)_{j=1}^\infty\in X.$ On the other hand, 
\[
 	t e_{n(\ell, k)}=e_{n(\ell, k)-t}=e_{n(\ell-t, k)},
 \]
thus $\theta_X$ is equivariant. The induced isomorphism $\tilde\theta_X$, after identification of Lemma \ref{realization}, is given by $$\langle\tilde\theta_X(T), \eta\rangle:=\langle T, \theta_X^{-1}(\eta)\rangle; \ \ \big(\eta\in \ell^2(\mathbb Z), T\in \mathscr{K}(L^2(X,\mu))\big),$$ 
and the $\mathbb Z$-action on $\mathscr{K}(L^2(X,\mu))$ is given by $t T(\xi)(x):=T(-t \xi)(-t x),$  for $t\in \mathbb Z, x\in X$, $\xi\in L^2(X, \mu)$, and $T\in \mathscr{K}(L^2(X,\mu))$, where $t\xi(x):=\xi(-t x)$, and likewise for the $\mathbb Z$-action on $\mathscr{K}(\ell^2(\mathbb Z))$. Now the fact that $\tilde\theta_X$ is equivariant follows directly from the same fact for $\theta_X$. Next, let us observe that $\tilde\theta_X$ is Cartan preserving. Indeed, we show the stronger assertion that the $\mathbb Z$-action fixes each point of the Cartan subalgebra $\mathscr{D}_X$. Again by Lemma \ref{realization}, it is enough to show that each function $f_\lambda$ is fixed by the $\mathbb Z$-action. For this, let $t\in \mathbb Z$ and $\lambda=(\lambda_{n{(\ell,k)}})\in c_0(\mathbb Z)$, and observe that,
\begin{align*}
f_\lambda((x_j),(y_j))&=\sum_{k\geq 1}\sum_{\ell=0}^{m_k} \lambda_{n(\ell, k)}\xi_\ell^k((x_j)) \bar\xi_\ell^k((y_j))\\&=\sum_{k\geq 1}\sum_{\ell=0}^{m_k} \lambda_{n(\ell, k)}{\rm exp}\big(-2\pi i(\ell+x_k)/m_k\big){\rm exp}\big(2\pi i(\ell+y_k)/m_k\big) \\&=\sum_{k\geq 1}\sum_{\ell=0}^{m_k} \lambda_{n(\ell, k)}{\rm exp}\big(2\pi i(y_k-x_k)/m_k\big),              
\end{align*}
hence,
\begin{align*}
	t f_\lambda((x_j),(y_j))&=f_\lambda(([x_j-t]_k),([y_j-t]_k))\\&=\sum_{k\geq 1}\sum_{\ell=0}^{m_k} \lambda_{n(\ell, k)}{\rm exp}\big(2\pi i([y_j-t]_k-[x_j-t]_k)/m_k\big)\\&=\sum_{k\geq 1}\sum_{\ell=0}^{m_k} \lambda_{n(\ell, k)}{\rm exp}\big(2\pi i([y_j-x_j]_k)/m_k\big)\\&=\sum_{k\geq 1}\sum_{\ell=0}^{m_k} \lambda_{n(\ell, k)}{\rm exp}\big(2\pi i(y_j-x_j)/m_k\big)= f_\lambda((x_j),(y_j)),             
\end{align*}
for each $x=(x_j), y=(y_j)\in X$, establishing the claim. On the other hand, $t\lambda=t(\lambda_{n{(\ell,k)}})=(\lambda_{n{([\ell-t]_k,k)}})$, which is again inside $c_0(\mathbb Z)$, since for $k^{'}$ large enough, $m_{k^{'}}\geq m_k+|t|$, and  for $0\leq \ell\leq m_k-1,$ we have, $0\leq \ell-t\leq m_k+|t|-1\leq m_{k'}-1$, thus $n{([\ell-t]_k,k)}=n([\ell-t]_{k^{'}},k^{'})=n(\ell-t,k^{'})$. Therefore, though the $\mathbb Z$-action does not fix the diagonal operators $D_\lambda$, it keeps the Cartan subalgebra $\mathscr{C}_X$ invariant. 

$(ii)$ For $Y=\mathbb Z_l\times lX$, we have $L^2(Y, \mu_l\times \nu)=\ell^2(\mathbb Z_l)\otimes L^2(lX, \nu)$, where $\mu_l$ is the counting measure on $\mathbb Z_l$, and the normalized left Haar measure  $\nu$ of  the clopen subgroup $lX$ of $X$ is simply the restriction of $\mu$. Now $\theta_X$ maps the subspace $L^2(lX, \nu)$ of $L^2(X,\mu)$ onto the subspace $\ell^2(l\mathbb Z)$. Therefore, $\theta_Y={\rm id}\otimes \theta_X$, where the first leg is the identity operator on $\ell^2(\mathbb Z_l)$. Likewise, the action of $\mathbb Z_l\times l\mathbb Z$ on $L^2(Y,\mu_l\times \nu)$ breaks into the translation action on $\mathbb Z_l$ and the restriction action of the subgroup $l\mathbb Z$ of $\mathbb Z$ on the invariant subspace $L^2(lX, \nu)$, thus all the assertions follow from part $(i)$.   
 \end{proof}  

\noindent {\bf Example 4.4 (continued).} For $l \in \mathbb{Z}$, consider the translation actions $\mathbb{Z}_l \times l\mathbb{Z} \curvearrowright \mathbb{Z}_l \times l\mathbb{Z}$ and $\mathbb{Z} \curvearrowright \mathbb{Z}$. Consider the bijection $\phi : \mathbb{Z} \rightarrow \mathbb{Z}_l \times l\mathbb{Z};\ \  (k + tl)\mapsto  (k,tl).$ Then $\mathbb{Z} \curvearrowright \mathbb{Z} \sim_{coe} \mathbb{Z}_l \times l\mathbb{Z} \curvearrowright \mathbb{Z}_l \times l\mathbb{Z}$, with respect to $\phi$ and the following continuous maps, 
\begin{align*}
&a : \mathbb{Z} \times \mathbb{Z} \rightarrow \mathbb{Z}_l \times l\mathbb{Z}; 
\ \ \big((k_1+t_1l), (k_2+t_2l)\big) \mapsto
\begin{cases}
(k_1,t_1l)  & k_1+k_2 < l \\
(k_1,t_1l+l) & k_1+k_2 \geq l,
\end{cases}
\\
&b : (\mathbb{Z}_l \times l\mathbb{Z}) \times (\mathbb{Z}_l \times \mathbb{Z}_l) \rightarrow  \mathbb{Z}; 
\ \ \big((k_1,t_1l), (k_2,t_2l)\big) \mapsto  
\begin{cases}
k_1+t_1l  & k_1+k_2 < l \\
k_1+t_1l-l & k_1+k_2 \geq l.
\end{cases}
\end{align*}
Applying the Proposition \ref{example-proposition} again, the corresponding actions of $\mathbb{Z}$ and $\mathbb{Z}_l \times l\mathbb{Z}$ on $G_l := \{ (lm +k , ln + k): k = 0,...,l-1, m,n \in \mathbb{Z}  \}$ and $H_l := \{ ((lm,k) , (ln , k)): k = 0,...,l-1, m,n \in \mathbb{Z}  \}$ are continuous orbit equivalent. These are discrete groupoids with $C^*_r(G_l) \cong C^*_r(G_l) \cong \bigoplus_{0}^{l-1} \mathscr{K}(\ell^2(\mathbb{Z}))$. Now Lemma \ref{equivariant} guarantees this to give equivariant \'{e}tale realization of the actions discussed in the beginning of this example.

Next, we give another non classical example of continuous orbit equivalence. First we need a lemma, which is implicit in \cite{li}.

\begin{lemma}\label{lemma example product}
In the notation of Example \ref{example1}, we have,
		
		$(i)$ $\lambda_n \times \lambda_m \sim_{\rm coe} \lambda_{mn}$,
		 
		 $(ii)$ for supernatural numbers $M,N$ and natural numbers $m,n$,
		\[((\mathbb{Z}_m \times m\mathbb{Z}) \curvearrowright H_M^m) \times ((\mathbb{Z}_n \times n\mathbb{Z}) \curvearrowright H_N^n) \sim_{\rm coe} ((\mathbb{Z}_{mn} \times mn\mathbb{Z}) \curvearrowright H_M^{mn}) \times (\mathbb{Z} \curvearrowright H_N^1).\]

\end{lemma}
\begin{proof}
	Consider the bijection $\phi : \mathbb{Z}_n \times \mathbb{Z}_m \rightarrow \mathbb{Z}_{nm};\ \ (t_n , t_m)\mapsto  t_m n + t_n$, and define the cocycle maps $a : (\mathbb{Z}_n \times \mathbb{Z}_m) \times (\mathbb{Z}_n \times \mathbb{Z}_m) \rightarrow \mathbb{Z}_{nm}$ and $b : \mathbb{Z}_{nm} \times \mathbb{Z}_{nm} \rightarrow (\mathbb{Z}_n \times \mathbb{Z}_m)$, defined by,
		\[
	a((\gamma_m , \gamma_n) , (t_m,t_n)) = 
	\begin{cases}
	(\gamma_m - m)n + (\gamma_n-n) & t_m + \gamma_m  \geq m ,\, t_n + \gamma_n  \geq n  \\
	\gamma_m n + (\gamma_n - n) & t_m + \gamma_m  < m ,\, t_n + \gamma_n  \geq n  \\
	(\gamma_m - m)n + \gamma_n & t_m + \gamma_m  \geq m ,\,t_n + \gamma_n  < n  \\
	\gamma_m n + \gamma_n & t_m + \gamma_m  < m ,\, t_n + \gamma_n  < n
	\end{cases}	
	\]
	and,
	\begin{align*}
	b(\gamma , t) =
	\begin{cases}
	\gamma_m  - m +1 , \gamma_n  - n  & \gamma  + t \geq mn ,\, \gamma_n + t_n  \geq  n \\
	(\gamma_m  - m, \gamma_n  ) & \gamma  + t \geq mn ,\, \gamma_n + t_n  <  n\\
	(\gamma_m  +1 , \gamma_n -n ) & \gamma  + t < mn ,\, \gamma_n + t_n  \geq  n \\
	(\gamma_m , \gamma_n  )  & \gamma  + t < mn ,\, \gamma_n + t_n  <  n
	\end{cases}
	\end{align*}
	for  $\gamma = \gamma_m n + \gamma_n$ and $t = t_m n +t_n$ in $\mathbb{Z}_{nm}$ and observe that,
		\begin{align*}
	\phi((\gamma_m , \gamma_n) (t_m,t_n)) &= 
	t_m n + t_n + (\gamma_m - m)n + (\gamma_n-n)\\&=
	a((\gamma_m , \gamma_n) , (t_m,t_n)) \phi (t_m,t_n),
	\end{align*}
	when $t_m + \gamma_m  \geq m , t_n + \gamma_n  \geq n$, and the same for the other three cases. Similarly, $
	\phi^{-1}(\gamma . t)=
	b(\gamma , t) \phi^{-1}(t)$, showing the  first statement. To see the second one, we use the identification $H_L^l$ and $\mathbb{Z}_L \times \tilde{X}_{lL}$ mentioned earlier. Let us identify actions $l\mathbb{Z} \curvearrowright \tilde{X}_{lL} = l\mathbb{Z} \curvearrowright l.(\mathbb{Z}_{lL})$ with  $\mathbb{Z} \curvearrowright \mathbb{Z}_L$ and observe that, for the diagonal action $\lambda^d_M: \mathbb Z\curvearrowright X_M\times X_M$, 
	
	\begin{align*}
	((\mathbb{Z}_m \times m\mathbb{Z}) \curvearrowright H_M^m ) \times ((\mathbb{Z}_n \times & n\mathbb{Z}) \curvearrowright H_N^n) =   \lambda_m \times \lambda_n \times \lambda^d_M\times\lambda^d_N 
	\\& \sim_{\rm coe} \lambda_{mn} \times \lambda^d_M\times\lambda^d_N
	\\&=
	((\mathbb{Z}_{mn} \times mn \mathbb{Z}) \curvearrowright \mathbb{Z}_{mn} \times (\tilde{X}_{mn} \times \tilde{X}_{mn})) \times \lambda^d_N
	\\&= 
	((\mathbb{Z}_{mn} \times mn \mathbb{Z}) \curvearrowright H_{M}^{mn}) \times (\mathbb{Z} \curvearrowright H_N^{1}).				
	\end{align*}

\end{proof}

\begin{example}[Product odometer transformation] \label{ot}
	Let $n_i , m_i \in \mathbb{N}$, and let $L_i$ be a supernatural number, for $i \in I:= \{1,2,\cdots,r\}$, and put $n:=\prod_{I} n_i = \prod_{I} m_i=:m$. Put $X := \prod_{I} \mathbb{Z}_{m_iL_i}$ and $Y := \prod_I \mathbb{Z}_{n_iL_i}$. The product odometer transformation actions $\mathbb{Z}^r \curvearrowright X$ and $\mathbb{Z}^r \curvearrowright Y$ are continuous orbit equivalent by Lemma \ref{lemma example product}. Therefore, 
	\begin{align*}
		\prod_I (\mathbb{Z}_{m_i} \times m_i \mathbb{Z}) \curvearrowright H^{m_i}_{L_i} & \sim_{\rm coe} (\mathbb{Z}_m \times \prod_I{m_i} \mathbb{Z}) \curvearrowright H_{L_1}^{m} \times \prod_{I\backslash\{1\}}  \mathbb{Z} \curvearrowright H_{L_i}^{1}
		\\&=
		(\mathbb{Z}_n \times \prod_I{n_i} \mathbb{Z}) \curvearrowright H_{L_1}^{n} \times \prod_{I\backslash\{1\}} \mathbb{Z} \curvearrowright H_{L_i}^{1}
		\\&\sim_{\rm coe} \prod_I (\mathbb{Z}_{n_i} \times n_i \mathbb{Z}) \curvearrowright H^{n_i}_{L_i}. 
	\end{align*}
Thus, $\mathbb{Z}^r \curvearrowright \prod_{I} G^{m_i}_{L_i} \sim_{\rm coe} \mathbb{Z}^r \curvearrowright \prod_{I} G^{n_i}_{L_i}$. By a similar argument as in Example \ref{example1}, $$C^*_r(G^{m_i}_{L_i}) \cong \bigotimes_{i = 1}^{r}  \bigoplus_{1}^{m_i} \mathscr{K}(L^2(\tilde{X}_{m_iL_i} , \mu_i)), \ C^*_r(G^{n_i}_{L_i}) \cong \bigotimes_{i = 1}^{r} \bigoplus_{1}^{n_i} \mathscr{K}(L^2(\tilde{X}_{n_iL_i} , \nu_i)),$$ and   $\mathbb{Z} \curvearrowright \bigotimes_{i = 1}^{r}  \bigoplus_{1}^{n_i} \mathscr{K}(L^2(\tilde{X}_{n_iL_i} , \nu_i))\sim_{\rm coe}\mathbb{Z} \curvearrowright \bigotimes_{i = 1}^{r}  \bigoplus_{1}^{m_i} \mathscr{K}(L^2(\tilde{X}_{m_iL_i} , \mu_i))$.
\end{example}

\vspace{.3cm}
Finally, let us give an example, as promised,  of non orbit equivalent actions with orbit equivalent (even conjugate) induced actions on \'{e}tale realizations. 

\begin{example} \label{example2}
	Consider the translation action $\mathbb Z\curvearrowright \mathbb Z$, and the action $\mathbb Z\curvearrowright \mathbb T$, defined by $n e^{i\theta}:=e^{i(\theta+n)},$ for $n\in \mathbb Z, \theta\in \mathbb R$. Then the actions  $\mathbb Z\curvearrowright \mathbb Z\times\mathbb Z$ and $\mathbb Z\curvearrowright \mathbb T\times\mathbb T$ could not be  continuous orbit equivalent  (as these equivalence relations are not topologically isomorphic), but they give rise to conjugate actions on their common \'{e}tale realization $\mathbb Z\times \mathbb Z$. We proceed as in Lemma \ref{equivariant} by constructing a $\mathbb Z$-equivariant isomorphism in the level of Hilbert spaces. Indeed, $\theta_{\mathbb T}: L^2(\mathbb T)\to L^2(\mathbb  Z)$ is noting but the Fourier-Plancherel transform $f\mapsto \hat f$, 
	$$\hat f(n):=\frac{1}{2\pi}\int_0^{2\pi} f(e^{i\theta})e^{-in\theta}d\theta, \ \ (n\in \mathbb Z, f\in L^2(\mathbb T)).$$
	The orthogonal basis of $L^2(\mathbb T)$ is $\hat{\mathbb T}=\mathbb Z$, given by $\xi_k(e^{i\theta})=e^{in\theta}$, for $k,n\in \mathbb Z, \theta\in \mathbb R$. Now it is easy to see that $\theta_{\mathbb T}$ is $\mathbb Z$-equivariant,
	\begin{align*}
		n\hat\xi_k(m)&=\hat\xi_k(m-n)=\frac{1}{2\pi}\int_0^{2\pi} \xi_k(e^{i\theta})e^{-i(m-n)\theta}d\theta\\&=\frac{1}{2\pi}\int_0^{2\pi} e^{-i(m-n-k)\theta}d\theta=\frac{1}{2\pi}\int_0^{2\pi} \xi_{n+k}(e^{i\theta})e^{-im\theta}d\theta=\hat\xi_{n+k}(m),		
	\end{align*}
	for $m,n, k\in\mathbb Z,$ that is,
	$n\theta_{\mathbb T}(\xi_k)=n e_k=e_{n+k}=\theta_{\mathbb T}(\xi_k),$
	as claimed. Let $\tilde\theta_{\mathbb T}: \mathscr{K}(L^2(\mathbb T))\to \mathscr{K}(\ell^2(\mathbb Z))$ be the   $\mathbb Z$-equivariant lift on compact operators and  for functions $f_\lambda$ defined by, 
	$$f_\lambda(e^{i\theta_1}, e^{i\theta_2}):=\sum_{k\in\mathbb Z}\lambda_k\xi_k(e^{i\theta_1})\bar\xi_k(e^{i\theta_2})=\sum_{k\in\mathbb Z}\lambda_k e^{i(\theta_1-\theta_2)}), \ \ (\lambda=(\lambda_k)\in c_0(\mathbb Z)),$$
	and for  $n\in \mathbb Z$ and $\theta_1, \theta_2\in \mathbb R$, observe that, 
	\[
	n f_\lambda(e^{i\theta_1}, e^{i\theta_2})=f_\lambda(e^{i(\theta_1-n)}, e^{i(\theta_2-n)})=f_\lambda(e^{i\theta_1}, e^{i\theta_2}),
	\]
	thus, $\mathbb Z$ fixes points of $\mathscr{D}_{\mathbb T}$, and it clearly also leaves $\mathscr{C}_{\mathbb Z}=c_0(\mathbb Z)$ invariant, showing that the induced actions $\mathbb Z\curvearrowright\mathbb Z\times \mathbb Z$ on the corresponding \'{e}tale realizations are conjugate.    
	
\end{example}

	\bibliographystyle{amsplain}

\end{document}